\documentclass[11pt]{article}

\usepackage{amsmath,amssymb}
\usepackage{color}
\usepackage{a4wide}
\usepackage{verbatim}
\usepackage{scrtime}
\usepackage{float}
\usepackage{graphicx}
\usepackage{graphics}
\usepackage{pstricks}
\usepackage{fancyhdr}
\usepackage{subfigure}
\usepackage{amsthm}
\usepackage{hyperref}
\usepackage[all]{hypcap}

\large\normalsize

\usepackage{enumitem}
\setitemize{noitemsep,topsep=0pt,parsep=0pt,partopsep=0pt}

\newtheorem{theorem}{Theorem}[section]

\newtheorem{proposition}[theorem]{Proposition}

\newcommand{\change}[1]{#1}

\newcommand{\trace}{{\rm Trace}}

\newcommand{\rc}{\nabla}
\newcommand{\D}{\mathrm{D}}
\newcommand{\PD}[1]{S_{++}({#1})}
\newcommand{\Stiefel}[2]{{\mathrm{St}({#1},{#2})}}
\newcommand{\symmetric}[1]{S_{sym}({#1})}
\newcommand{\skewSymmetric}[1]{S_{skew}({#1})}
\newcommand{\OG}[1]{{\mathcal{O}({#1})}}
\newcommand{\Sym}{{\mathrm{Sym}}}
\newcommand{\Skew}{{\mathrm{Skew}}}
\newcommand{\argmin}{\operatornamewithlimits{arg\,min}}
\newcommand{\grad}{\mathrm{grad}}
\newcommand{\hess}{\mathrm{Hess}}
\newcommand{\Grad}{\mathrm{Grad}}
\newcommand{\mat}[1]{{\bf #1}}
\newcommand{\subject}{\mathrm{subject\  to}}
\newcommand*\samethanks[1][\value{footnote}]{\footnotemark[#1]}

\title{Low-rank optimization with trace norm penalty\thanks{This paper presents research results of the Belgian Network DYSCO (Dynamical Systems, Control, and Optimization), funded by the Interuniversity Attraction Poles Programme, initiated by the Belgian State, Science Policy Office. The scientific responsibility rests with its authors.}}

\author{B.~Mishra\thanks{Department of Electrical Engineering and Computer Science, University of Li\`ege, 4000 Li\`ege,
Belgium (B.Mishra@ulg.ac.be, G.Meyer@ulg.ac.be, R.Sepulchre@ulg.ac.be). Bamdev Mishra is a research fellow of the Belgian National Fund for Scientific Research (FNRS).}
        \and G.~Meyer\samethanks[2]
        \and F.~Bach\thanks{INRIA - Sierra Project-Team Ecole Normale Sup\'erieure Paris, France (Francis.Bach@inria.fr)}
        \and R.~Sepulchre\samethanks[2]
        }

\begin{document}

\maketitle

\begin{abstract}
The paper addresses the problem of low-rank trace norm minimization. We propose an algorithm that alternates between fixed-rank optimization and rank-one updates. The fixed-rank optimization is characterized by an efficient factorization that makes the trace norm differentiable in the search space and the computation of duality gap numerically tractable. The search space is nonlinear but is equipped with a particular Riemannian structure that leads to efficient computations. We present a second-order trust-region algorithm with a guaranteed quadratic rate of convergence. Overall, the proposed optimization scheme converges super-linearly to the global solution while maintaining complexity that is linear in the number of rows \change{and columns} of the matrix. To compute a set of solutions efficiently for a grid of regularization parameters we propose a predictor-corrector approach that outperforms the naive warm-restart approach on the \change{fixed-rank} quotient manifold. The performance of the proposed algorithm is illustrated on problems of low-rank matrix completion and multivariate linear regression.
\end{abstract}


\pagestyle{myheadings}
\thispagestyle{plain}
\markboth{B. MISHRA, G. MEYER, F. BACH AND R. SEPULCHRE}{LOW-RANK OPTIMIZATION WITH TRACE NORM PENALTY}

\section{Introduction}
The present paper focuses on the convex program
\begin{equation}\label{eq:general_formulation}
\begin{array}{l c}
\min\limits_{\mat{X} \in \mathbb{R}^{n \times m}} \quad f(\mat{X})  + \lambda \| \mat{X} \|_*\\
\end{array}
\end{equation}
where $f$ is a smooth convex function, $\| \mat{X} \|_*$ is the \emph{trace norm} (also known as nuclear norm) which is the sum of the singular values of $\mat{X}$ \cite{fazel02a, recht10a, cai10a} and $\lambda > 0$ is the regularization parameter. Programs of this type have attracted much attention in the recent years as efficient convex relaxations of intractable rank minimization problems \cite{fazel02a}. The rank of the optimal solution $\mat{X}^*(\lambda)$ of (\ref{eq:general_formulation}) decreases to zero as the regularization parameter grows unbounded \cite{bach08a}. As a consequence, generating efficiently the regularization path $\{ \mat{X}^*(\lambda_i)\}_{i = 1,...,N} $, for a whole range of values of $\lambda_i$ minimizers, is a convenient proxy to obtain suboptimal low-rank \change{minimizers} of $f$.

Motivated by machine learning and statistical large-scale regression problems \cite{recht10a, yuan07a, vounou10a}, we are interested  in very low-rank solutions ($p < 10^2$) of very high-dimensional problems ($n > 10^6$). To this end, we propose an algorithm that guarantees second-order convergence to the solutions of (\ref{eq:general_formulation}) while ensuring a tight control (linear in $n$) on the data storage requirements and on the numerical complexity of each iteration.

The proposed algorithm is based on a low-rank factorization of the unknown matrix, similar to the singular value decomposition (SVD), $\mat{X} = \mat{U} \mat{B} \mat{V}^T$. Like in SVD, $\mat{U} \in \mathbb{R}^{n \times p}$ and $\mat{V} \in \mathbb{R}^{m\times p}$ are orthonormal matrices that span row and column spaces of $\mat{X}$. In contrast, the $p\times p$ scaling factor $\mat{B} = \mat{B}^T \succ 0$ is allowed to be non-diagonal which makes the factorization non-unique.

Our algorithm alternates between fixed-rank optimization and rank-one updates. When the rank is fixed, the problem is no longer convex but the search space has nevertheless a Riemannian structure. We use the framework of manifold optimization to devise a trust-region algorithm that generates low-cost (linear in $n$) iterates that converge super-linearly to a local minimum. Local minima are escaped by incrementing the rank until the global minimum in reached. The rank-one update is always selected to ensure a decrease of the cost. 

Implementing the complete algorithm for a fixed value of the regularization parameter $\lambda$ leads to a monotone convergence to the global minimum through a sequence of local minima of increasing ranks. Instead, we also modify $\lambda$ along the way with a predictor-corrector method thereby transforming most local minima of (\ref{eq:general_formulation}) (for fixed $\lambda$ and fixed rank) into global minima of (\ref{eq:general_formulation}) for different values of $\lambda$. The resulting procedure, thus, provides a full regularization path at a very efficient numerical cost.

Not surprisingly, the proposed approach has links with several earlier contributions in the literature. Primarily, the idea of interlacing fixed-rank optimization with rank-one updates has been used in semidefinite programming \cite{burer03a, journee10a}. It is here extended to a non-symmetric framework using the Riemannian geometry recently developed in \cite{bonnabel09a, meyer11b, meyer11a}. An improvement with respect to the earlier work \cite{burer03a, journee10a} is the use of duality gap certificate to discriminate between local and global minima and its efficient computation thanks to the chosen parameterization.

Schemes that combine \change{fixed-rank} optimization and special rank-one updates have appeared recently in the particular context of matrix completion \cite{keshavan09b, wen12a}. The framework presented here is in the same spirit but in a more general setting and with a global convergence analysis. Most other fixed-rank algorithms \cite{srebro03a, keshavan09b, meka09a, simonsson10a, wen12a, meyer11b, boumal11b, vandereycken13a} for matrix completion are first-order schemes. It is more difficult to provide a tight comparison of the proposed algorithm to trace norm minimization algorithms that do not fix the rank a priori \cite{cai10a, mazumder10a, yuan07a, amit07a}. It should be emphasized, however, that most trace norm minimization algorithms use singular value thresholding operation at each iteration. This is the most numerically demanding step for these algorithms. For the matrix completion application, it involves computing (potentially all) the singular values of a \emph{low-rank} + \emph{sparse} matrix \cite{cai10a}.  \change{In contrast, the proposed approach requires only dense linear algebra (linear in $n$) and rank-one updates using only dominant singular vectors and value of a sparse matrix. The main potential of the algorithm appears when computing the solution not for a single parameter $\lambda$ but for a number of values of $\lambda$. We compute the entire regularization path with an efficient predictor-corrector strategy that convincingly outperforms the warm-restart strategy.} 

For the sake of illustration and empirical comparison with state-of-the-art algorithms we consider two particular applications, low-rank matrix completion \cite{candes09a} and multivariate linear regression \cite{yuan07a}. In both cases, we obtain iterative algorithms with a numerical complexity that is linear in the number of observations and with favorable convergence and precision properties.

 \section{Relationship between convex program and non-convex formulation}
Among the different \change{factorizations} that exist to represent low-rank matrices, we use the \change{factorization} \cite{meyer11a, bonnabel09a} that decomposes a rank-$p$ matrix $\mat{X} \in \mathbb{R}^{n \times m}$ into
\begin{equation*}
 \mat{X} = \mat{U} \mat{B} \mat{V}^T
 \end{equation*}
where $\mat{U} \in \Stiefel{p}{n}$, $\mat{V} \in \Stiefel{p}{m}$ and $\mat{B} \in \PD{p}$. $\Stiefel{p}{n}$ is the Stiefel manifold or the set of $n\times p$ matrices with orthonormal columns. $\PD{p}$ is the cone of $p\times p$ positive definite matrices. We stress that the scaling $\mat{B} = \mat{B}^T \succ 0$ is not required to be diagonal. The redundancy of this parameterization has non-trivial algorithmic implications (see Section \ref{sec:manifold_optimization}) but we believe that it is key to success of the approach. See \cite{keshavan09b,meyer11a} for earlier algorithms advocating matrix scaling and Section \ref{sec:polar_vs_svd} for a numerical illustration. With the use of \change{factorization} $\mat{X} = \mat{UBV}^T$, the trace norm is written as $\| \mat{X}\|_* = \trace(\mat{B})$ which makes it \change{differentiable}. \change{For a fixed rank $p$, the optimization problem (\ref{eq:general_formulation}) is recast as}
\begin{equation}\label{eq:factorized_formulation}
\begin{array}{ll}
\min\limits_{\mat{U},\mat{B}, \mat{V}} & f(\mat{UBV}^T) + \lambda \trace(\mat{B}) \\
\subject & \mat{U} \in \Stiefel{p}{n}, \quad \mat{B} \in \PD{p} \quad {\rm and} \quad \mat{V} \in \Stiefel{p}{m}.\\
\end{array}
\end{equation}
The search space of (\ref{eq:factorized_formulation}) is not Euclidean but the product space of two well-studied manifolds, namely, the Stiefel manifold \cite{edelman98a} and the cone of positive definite matrices \cite{smith05a}. This provides a proper geometric framework to perform optimization. From the geometric point of view, the column and row spaces of $\mat{X}$ are represented on the Stiefel manifold whereas the scaling factor is absorbed into the positive definite part. A proper metric on the space takes into account both rotational and scaling invariance.

\subsection{First-order optimality conditions}
In order to relate the fixed-rank problem (\ref{eq:factorized_formulation}) to the convex optimization problem (\ref{eq:general_formulation}) we look at the necessary and sufficient optimality conditions that govern the solutions. The first-order necessary and sufficient optimality condition for the convex program (\ref{eq:general_formulation}) is 
 \begin{equation}\label{eq:foc_general_formulation}
 \mat{0} \in \Grad_\mat{X} f(\mat{X}) + \lambda \partial \| \mat{X} \|_*
 \end{equation}
 where $\Grad_{\mat{X}} f$ is the Euclidean gradient of $f$ in $\mathbb{R}^{n \times m}$ at $\mat{X}$ and $\partial \| \mat{X} \|_*$ is the sub-differential of the trace norm (optimality conditions for trace norm are in \cite{bach08a, recht10a}). 
\begin{proposition}\label{prop:foc_factorized_formulation}
\change{The first-order necessary optimality conditions of (\ref{eq:factorized_formulation})}
 \begin{equation}\label{eq:foc_factorized_formulation}
 \begin{array}{lll}
\mat{S}\mat{VB} - \mat{U}\Sym(\mat{U}^T \mat{SVB})&= & \mat{0}\\
\Sym( \mat{U}^T \mat{S} \mat{V}   + \lambda \mat{I} )  & = &\mat{0}\\
\mat{S}^T \mat{UB} -\mat{V} \Sym (\mat{V}^T \mat{S}^T \mat{UB} )&= &\mat{0} 
 \end{array}
 \end{equation}
where \change{$\mat{X} = \mat{UBV}^T$}, $\Sym(\mat{\Delta}) = \frac{\mat{\Delta} + \mat{\Delta}^T}{2}$ for any square matrix $\mat{\Delta}$ and $\mat{S}=\Grad_\mat{X} f(\mat{UBV}^T) $. $\mat{S}$ is referred to as \emph{dual variable} throughout the paper.
\end{proposition} 
\begin{proof}
The first-order optimality conditions are derived either by writing the Lagrangian of the problem (\ref{eq:factorized_formulation}) and looking at the \emph{KKT conditions} or by deriving the gradient of the function on the structured space $\Stiefel{p}{n }\times \PD{p} \times \Stiefel{p}{m }$ using the metric (\ref{eq:metric}) defined in Section \ref{sec:manifold_optimization}. \change{The proof is given in Appendix \ref{appendix:foc}.}
\end{proof}

 \begin{proposition}\label{prop:convergence}
 A local minimum of (\ref{eq:factorized_formulation}) $\mat{X} = \mat{U} \mat{B} \mat{V}^T$ is also the global optimum of (\ref{eq:general_formulation}) iff $ \| \mat{S} \|_{op} = \lambda$ where $\mat{S}=\Grad_\mat{X} f(\mat{UBV}^T) $ and $\| \mat{S}\|_{op}$ is the operator norm, i.e., the dominant singular value of $\mat{S}$. Moreover, $\| \mat{S} \|_{op} \geq \lambda$ and equality holds only at optimality.
 \end{proposition}
 \begin{proof}
 This is in fact \change{rewriting} the first-order optimality condition of \change{(\ref{eq:general_formulation})} \cite{cai10a,ma11a}. \change{The proof is given in Appendix \ref{appendix:foc_convex}.}
 \end{proof}

A local minimum of (\ref{eq:factorized_formulation}) is identified with the global minimum of (\ref{eq:general_formulation}) if \\$\| \mat{S}\|_{op} - \lambda \leq \epsilon$ where $\epsilon$ is a user-defined threshold.

\subsection{Duality gap computation} \label{sec:duality_gap}
Proposition \ref{prop:convergence} provides a criterion to check the global optimality of a solution of (\ref{eq:factorized_formulation}). Here however, it provides no guarantees on \emph{closeness} to the global solution. A better way of certifying closeness for the optimization problem of type (\ref{eq:general_formulation}) is provided by the duality gap. The duality gap characterizes the difference of the obtained solution from the optimal solution and is always non-negative \cite{boyd04a}. 

\begin{proposition}\label{prop:dual_formulation}
\change{The Lagrangian dual formulation of (\ref{eq:general_formulation}) is}
\begin{equation}\label{eq:dual_formulation}
\begin{array}{lll}
\max\limits_{\mat M}& - f^*(\mat{M})\\
 \subject &  \| \mat{M} \|_{op} \leq \lambda 
\end{array}
\end{equation}
where $\mat{M} \in \mathbb{R}^{n \times m}$ is the dual variable, $\| \mat{M} \|_{op}$ is the largest singular value of $\mat{M}$ and is the dual norm of the trace norm. $f^*$ is the Fenchel conjugate \cite{bach11a, boyd04a} of $f$, defined as 
$
f^*(\mat{M}) = {\rm sup}_{\mat{X} \in \mathbb{R}^{n \times m}} \left [ \trace(\mat{M}^T \mat{X})  - f(\mat{X}) \right ].
$
\end{proposition}
\begin{proof}
\change{The proof is given in Appendix \ref{appendix:dual_formulation}.}
\end{proof}

When \change{$\| \mat{M} \|_{op} \leq \lambda$}, the expression of duality gap is 
\begin{equation} \label{eq:duality_gap}
f(\mat{X}) + \lambda \| \mat{X}\|_*  + f^*(\mat{M}) 
\end{equation}
where $\mat{M}$ is the \emph{dual candidate}. A good choice for the dual candidate $\mat{M}$ is $\mat{S}\  (= \Grad_{\mat{X}} f(\mat{X}))$ with appropriate scaling to satisfy the operator norm constraint: $\mat{M} = \min\{1, \frac{\lambda }{\| \mat{S} \|_{op}} \} \mat{S}$ \cite{bach11a}. As an extension for some functions $f$ of type $f(\mat{X}) = \psi(\mathcal{A} (\mat{X}))$ where $\mathcal{A}$ is a linear operator, computing the Fenchel conjugate of the function $\psi$ may be easier than that of $f$.  When $\| \mathcal{A}^*( \mat{M}) \|_{op} \leq \lambda$ the duality gap, using similar calculations as above, is $f(\mat{X}) + \lambda \| \mat{X}\|_*  + \psi^*(\mat{M}) \quad {\rm when}$ where $\mathcal{A}^*$ is the adjoint operator of $\mathcal{A}$ and $\psi^*$ is the Fenchel conjugate of $\psi$. A good choice of $\mat{M}$ is again $\min \{1, \frac{\lambda }{\sigma_\psi}\} \Grad \psi$ where $\sigma_\psi$ is the dominant singular value of $\mathcal{A}^{*}(\Grad \psi)$ \cite{bach11a}.

\section{Manifold-based optimization to solve the non-convex problem (\ref{eq:factorized_formulation})}\label{sec:manifold_optimization}
In this section we \change{propose an algorithm} to obtain a local minimum for the problem (\ref{eq:factorized_formulation}). In contrast to first-order optimization algorithms proposed earlier in \cite{meyer11a, meyer10a, keshavan09b}, we develop a second-order trust-region algorithm that has a quadratic rate of convergence \cite{nocedal06a, absil08a}. The idea behind a trust-region algorithm is to build locally a quadratic model of the function at a point and solve the \emph{trust-region subproblem} to get the next potential iterate. Depending on whether the decrease in the objective function is sufficient or not, the potential iterate is accepted or rejected. Details about a general trust-region algorithm are given in \cite{nocedal06a}. We rewrite (\ref{eq:factorized_formulation}) as
\begin{equation}\label{eq:factorized_formulation_two_var}
\begin{array}{ll}
\min\limits_{\mat{U}, \mat{B}, \mat{V}} & \bar{\phi}(\mat{U}, \mat{B}, \mat{V}) \\
\subject & (\mat{U},\mat{B},\mat{V})  \in \Stiefel{p}{n} \times \PD{p} \times \Stiefel{p}{m} \\
\end{array}
\end{equation}
where $\bar{\phi}(\mat{U}, \mat{B}, \mat{V}) = f(\mat{U}\mat{B}\mat{V}^T) + \lambda \trace(\mat{B})$ is introduced for notational convenience. An important observation for second-order algorithms \cite{absil09a, absil08a} is that the local minima of the problem (\ref{eq:factorized_formulation_two_var}) are not isolated in the search space 
\[
\overline{\mathcal{M}}_p =  \Stiefel{p}{n} \times \PD{p} \times \Stiefel{p}{m}.
\]
This is because the cost function is invariant under rotations, $\mat{U}\mat{B}\mat{V}^T =\\ (\mat{UO})(\mat{O}^T\mat{BO}) (\mat{VO})^T$ for any $p\times p$ rotation matrix $\mat{O} \in \OG{p}$. \change{Note that $\OG{p}$ takes away all the symmetry of the total space. This is done by counting the dimension of the quotient space\footnote{The dimension of the total space is $(np - \frac{p(p+1)}{2}) + (mp -\frac{p(p+1)}{2}) + \frac{p(p+1)}{2}$ and that of $\OG{p}$ is $\frac{p(p-1)}{2}$. Hence, the dimension of the quotient space is equal to dimension of total space $-$ dimension of $\OG{p}$.} which is $(n + m -p)p$ . This is same as dimension of the rank-$p$ matrices.}

To remove the symmetry of the cost function, we identify all the points of the search space that belong to the equivalence class defined by
\[
[(\mat{U}, \mat{B}, \mat{V})]  =  \{   (\mat{U}\mat{O},\mat{O}^T \mat{B} \mat{O}, \mat{V}\mat{O}  )  | \mat{O} \in \OG{p}  \}.
\]
The set of all such equivalence classes is denoted by
\begin{equation}\label{eq:quotient_manifold}
\mathcal{M}_p = \overline{\mathcal{M}}_p / \OG{p}
\end{equation}
\change{which has the structure of a smooth quotient manifold  $\overline{\mathcal{M}}_p$ by $\OG{p}$ \cite[ Theorem $9.16$]{lee03a}.} Problem (\ref{eq:factorized_formulation_two_var}) is thus conceptually an \emph{unconstrained} optimization problem on the quotient manifold $\mathcal{M}_p$ in which the minima are isolated. Computations are performed in the total space $\overline{\mathcal{M}}_p$, which is the product space of well-studied manifolds.
\begin{figure}[ht]
\centering
\includegraphics[scale = 0.45]{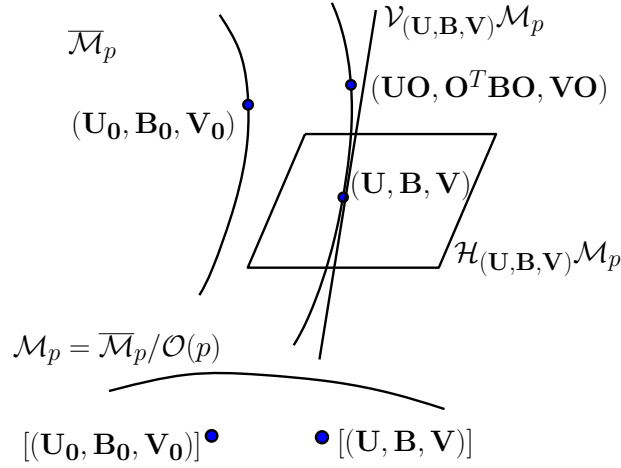}
\caption{The quotient manifold representation of the search space.}
\label{fig:quotient_manifold}
\end{figure}

\subsection*{Tangent space of $\mathcal{M}_p$}
Tangent vectors at a point $x \in \mathcal{M}_p$ have a matrix representation in the tangent space of the total space $\overline{\mathcal{M}}_p$. Note that $\bar{x}$ belongs to $\overline{\mathcal{M}}_p$ and its equivalence class is represented by the element $x\in \mathcal{M}_p$ such that $x = [\bar{x}]$. Because the total space is a product space $\Stiefel{p}{n} \times \PD{p} \times \Stiefel{p}{m}$, its tangent space admits the decomposition at a point $\bar{x} = (\mat{U}, \mat{B}, \mat{V})$
\[
T_{\bar{x}} \overline{\mathcal{M}}_p =  T_{\mat{U}} \Stiefel{p}{n} \times T_{\mat{B}}\PD{p} \times T_{\mat{V}} \Stiefel{p}{m}
\]
and the following characterizations are well-known \cite{edelman98a,smith05a}
\[
\begin{array}{lll}
T_{\mat{U}} \Stiefel{p}{n} &= & \{ \mat{Z}_\mat{U} -  \mat{U}\Sym(\mat{U}^T \mat{Z}_{\mat{U}})\  |\   \mat{Z}_\mat{U}  \in \mathbb{R}^{n  \times p} \}\\
T_{\mat{B}}\PD{p} & =& \symmetric{p}
\end{array}
\]
where $\symmetric{p}$ is the set of $p\times p$ symmetric matrices.

Note that an arbitrary matrix $(\mat{Z}_{\mat{U}}, \mat{Z}_{\mat{B}}, \mat{Z}_{\mat{V}}) \in \mathbb{R}^{n \times p} \times \mathbb{R}^{p \times p} \times \mathbb{R}^{m \times (p)}$ is projected on the tangent space $T_{\bar{x}} \overline{\mathcal{M}}_p$ by the linear operation
\begin{equation}\label{eq:projection_tangent_space}
\begin{array}{ll}
\Psi_{\bar{x}}  (  \mat{Z}_\mat{U}, \mat{Z}_\mat{B}, \mat{Z}_\mat{V} ) =  ( \mat{Z}_\mat{U} - \mat{U} \Sym(\mat{U}^T \mat{Z}_\mat{U}), \Sym{(\mat{Z}_{\mat {B}})},  \mat{Z}_\mat{V} - \mat{V} \Sym(\mat{V}^T \mat{Z}_\mat{V}) ).
  \end{array}
\end{equation}
\change{where $\Sym(\mat{Z}_{\mat B}) =  (\mat{Z}_{\mat B} + \mat{Z}^T_{\mat B})/2$}. A matrix representation of the tangent space at $x \in \mathcal{M}_p$ relies on the decomposition of $T_{\bar{x}} \overline{\mathcal{M}}_p$ into its \emph{vertical} and \emph{horizontal} subspaces. The vertical space $\mathcal{V}_{\bar{x}}\mathcal{M}_p$ is the subspace of $T_{\bar{x}} \overline{\mathcal{M}}_p$ that is tangent to the equivalence class $[\bar{x}]$
\begin{equation}\label{eq:vertical_space}
\mathcal{V}_{\bar{x}} \mathcal{M}_p =  \{  (   \mat{U}\mat{\Omega} ,  \mat{B\Omega} - \mat{\Omega B}, \mat{V}\mat{\Omega}   )\  |\ \mat{\Omega} \in \skewSymmetric{p}  \}
\end{equation}
\change{where $\skewSymmetric{p}$ is the set of skew symmetric matrices of size $p \times p$}. The horizontal space $\mathcal{H}_{\bar{x}}\mathcal{M}_p$ must be chosen such that $T_{\bar{x}} \overline{\mathcal{M}}_p =  \mathcal{H}_{\bar{x}} \mathcal{M}_p \oplus  \mathcal{V}_{\bar{x}} \mathcal{M}_p$. We choose $\mathcal{H}_{\bar{x}}\mathcal{M}_p$ as the orthogonal complement of $\mathcal{V}_{\bar{x}}\mathcal{M}_p$ for the metric
\begin{equation}\label{eq:metric}
\begin{array}{lll}
\bar{g}_{\bar{x}}
 (      \xi_{\bar{x}} ,   \eta_{\bar{x}}  )  & =  &\trace  (     \xi_{\mat U}^T  \eta_{\mat U})   +  \trace ( \mat{B}^{-1} \xi_{\mat B} \mat{B}^{-1} \eta_\mat{B}  ) +  \trace (     \xi_{\mat V}^T  \eta_{\mat V} ),
\end{array}
\end{equation}
which picks the normal metric of the Stiefel manifold \cite{edelman98a} and the natural metric of the positive definite cone \cite{smith05a}. Here $\xi_{\bar{x}}$ and $\eta_{\bar{x}}$ are elements of $T_{\bar{x}} \mathcal{M}_p$. With this choice, a horizontal tangent vector $\zeta_{\bar{x}}$ is any tangent vector $(\zeta_{\mat{U}}, \zeta_{\mat{B}}, \zeta_{\mat{V}})$ belonging to the set
\begin{equation}\label{eq:horizontal_space}
\begin{array}{lll}
\mathcal{H}_{\bar{x}} \mathcal{M}_p & =  &  \{    (  \zeta_{\mat{U}} , \zeta_\mat{B}, \zeta_{\mat{V}}   )  \in  T_{\bar{x}} \overline{\mathcal{M}}_p \ | \  \bar{g}_{\bar{x}}    (  (\zeta_\mat{U}, \zeta_{\mat{B}},\zeta_{\mat{V}}  ), (  \mat{U}\mat{\Omega} ,  ( \mat{B\Omega} - \mat{\Omega B} ) , \mat{V}\mat{\Omega}  )      )  = 0   \} \\
 \end{array}
\end{equation}
\change{for all $\mat{\Omega} \in \skewSymmetric{p}$. Another characterization of the horizontal space is $\mathcal{H}_{\bar{x}} \mathcal{M}_p =    \{    (  \zeta_{\mat{U}} , \zeta_\mat{B}, \zeta_{\mat{V}}   )  \in  T_{\bar{x}} \overline{\mathcal{M}}_p\  |\  \left(\zeta_{\mat{U}}^T\mat{U} + \mat{B}^{-1}\zeta_{\mat B}  - \zeta_{\mat B} \mat{B}^{-1} +  \zeta_{\mat{V}}^T\mat{V} \right) \rm{\  is\ symmetric } 
  \}$. The horizontal space is invariant by the group action along the equivalence class.} Starting from an arbitrary tangent vector $\eta_{\bar{x}} \in T_{\bar{x}} \overline{\mathcal{M}}_p$ we construct its projection on the horizontal space by picking $\mat{\Omega} \in \skewSymmetric{p}$ such that
\begin{equation}\label{eq:projection}
\Pi_{\bar{x}}(\eta_{\bar{x}})=  
\begin{array}{ll}
( \eta_{\mat{U}} - \mat{U}\mat{\Omega} , \eta_{\mat{B}} - (\mat{B\Omega} - \mat{\Omega B}),  \eta_{\mat{V}} - \mat{V}\mat{\Omega}   )
\end{array}
 \in \mathcal{H}_{\bar{x}} \mathcal{M}_p,
\end{equation}
Using the calculation (\ref{eq:horizontal_space}), the unique $\mat{\Omega}$ that satisfies (\ref{eq:projection}) is the solution of the \emph{Lyapunov} equation
\begin{equation}\label{eq:Lyapunov}
\begin{array}{llll}
\mat{\Omega} \mat{B}^2 + \mat{B}^2 \mat{\Omega} = \mat{B} ( \Skew(\mat{U}^T \eta_{\mat{U}}) - 2 \Skew(\mat{B}^{-1} \eta_{\mat {B}})+ \Skew(\mat{V}^T \eta_{\mat{V}})    ) \mat{B}
\end{array}
\end{equation}
where $\Skew(\mat{A}) = (\mat{A} - \mat{A}^T)/2$ and \change{$(\eta_{\mat U}, \eta_{\mat B}, \eta_{\mat V})$ is the matrix representation of $\eta_{\bar x}$}. The numerical complexity of solving the Lyapunov equation is $O(p^3)$ \cite{bartels72a}.

\subsection*{The Riemannian submersion $(\mathcal{M}_p, g)$}
\change{The choice of the metric (\ref{eq:metric}), which is invariant along the equivalence class $[\bar{x}]$ turns the quotient manifold $\mathcal{M}_p$ into a Riemannian submersion of $(\overline{\mathcal{M}}_p, \bar{g})$ \cite[Theorem $9.16$]{lee03a} and \cite[Section $3.6.2$]{absil08a}}. As shown in \cite{absil08a}, this special construction allows for a convenient matrix representation of the gradient \cite[Section $3.6.2$]{absil08a} and the Hessian \cite[Proposition $5.3.3$]{absil08a} on the abstract manifold $\mathcal{M}_p$. The Riemannian gradient of $\phi: \mathcal{M}_p \rightarrow \mathbb{R}: x \mapsto \phi(x) = \bar{\phi}(\bar{x})$ is uniquely represented by its horizontal lift in $\overline{\mathcal{M}}_p$ which has the matrix representation
\begin{equation}\label{eq:Riemannian_gradient}
\overline { {\grad}_x \phi} = \grad_{\bar{x}} \bar{\phi}.
\end{equation}
It should be emphasized that $\grad_{\bar{x}} \bar{\phi}$ is in the the tangent space $T_{\bar{x}} \overline{\mathcal{M}}_p$. However, due to invariance of the cost \change{function} along the equivalence class $[\bar{x}]$, $\grad_{\bar{x}} \bar{\phi}$ also belongs to the horizontal space $\mathcal{H}_{\bar{x}} \mathcal{M}_p$ and hence, the equality in (\ref{eq:Riemannian_gradient}) \cite{absil08a}. \change{The matrix expression of $\grad_{\bar{x}} \bar{\phi}$ in the total space $\overline{\mathcal{M}}_p$ at a point $\bar{x} = (\mat{U}, \mat{B}, \mat{V})$ is obtained from its definition: it is unique element of $T_{\bar x} \overline{\mathcal M}_{p}$ that satisfies ${\rm D} \bar{\phi}[\eta_{\bar x}] = \bar{g}_{\bar{x}}(     \grad_{\bar x} \bar{\phi},   \eta_{\bar{x}}  )$ for all $ \eta_{\bar x} \in T_{\bar x} \overline{\mathcal M}_{p}$. ${\rm D} \bar{\phi}[\eta_{\bar x}]$ is the standard Euclidean directional derivative of $\bar{\phi}$ in the direction $\eta_{\bar x}$. This definition leads to the matrix representations
}
\begin{equation}\label{eq:riemannian_gradient}
\begin{array}{lll}
\grad_{\mat{U}} \bar{\phi}= \grad_{\mat U} \bar{\phi}_{\rm Euclidean},\ 
\grad_{\mat{B}} \bar{\phi}= \mat{B} \left( \grad_{\mat B} \bar{\phi}_{\rm Euclidean} \right) \mat{B}\\
\grad_{\mat{V}} \bar{\phi}= \grad_{\mat V} \bar{\phi}_{\rm Euclidean}\\
\end{array}
\end{equation}
\change{where $\grad_{\bar{x}}\bar{\phi}_{\rm Euclidean}$ is $\Psi_{\bar{x}}(\Grad_{\mat{U}} \bar{\phi}, \Grad_{\mat{B}} \bar{\phi}, \Grad_{\mat{V}} \bar{\phi})$ and $(\Grad_{\mat{U}} \bar{\phi}, \Grad_{\mat{B}} \bar{\phi}, \Grad_{\mat{V}} \bar{\phi})$ is the gradient of $\bar{\phi}$ in the Euclidean space $\mathbb{R}^{n \times r} \times \mathbb{R}^{r \times r} \times \mathbb{R}^{m \times r}$. Here $\Psi_{\bar{x}}$ is projection operator (\ref{eq:projection_tangent_space}).}

Likewise, the Riemannian connection $\rc_\nu \eta$ on $\mathcal{M}_p$ is uniquely represented by its horizontal lift in $\overline{\mathcal{M}}_p$ which is $\overline { {\rc}_\nu \eta} = \Pi_{\bar{x}} (\overline{\rc}_{\bar{\nu}} \bar{\eta})$ where $\nu$ and $\eta$ are vector fields in $\mathcal{M}_p$ and $\bar{\nu}$ and $\bar{\eta}$ are their horizontal lifts in $\overline{\mathcal{M}}_p$. Once again, the Riemannian connection $\overline{\rc}_{\bar{\nu}} \bar{\eta}$ on $\overline{\mathcal{M}}_p$ has well-known expression \cite{journee09a, smith05a, absil08a}, obtained by means of the \emph{Koszul formula}. \change{The Riemannian connection on the Stiefel manifold is derived in \cite[Example $4.3.6$]{journee09a} and on the positive definite cone is derived in \cite[Appendix B]{meyer11b}. Finally, the Riemannian connection on the total space is given by}
\begin{equation}\label{eq:riemannian_connection}
\begin{array}{lll}
\overline{\rc}_{\bar{\nu}} \bar{\eta} = \Psi_{\bar{x}}(\D \bar{\eta}  [\bar{\nu}] ) -  
\begin{array}{ll}
\Psi_{\bar{x}}( \nu_{\mat{U}}\Sym(\mat{U}^T \eta_{\mat{U}}),   \Sym(\nu_{\mat{B}} \mat{B}^{-1} \eta_{\mat{B}}) ,\nu_{\mat{V}}\Sym(\mat{V}^T \eta_{\mat{V}}) )
\end{array}
\end{array}
\end{equation}
Here $\D \bar{\eta} [\bar{\nu}]$ is the classical Euclidean directional derivative of $\bar{\eta}$ in the direction $\bar{\nu}$. The Riemannian Hessian in $\overline{\mathcal{M}}_p$ has, thus, the following matrix expression
\begin{equation}\label{eq:riemannian_hessian}
\overline{\hess \phi(x) [\xi]}= \Pi_{\bar{x}}(  \overline{\rc}_{\bar{\xi}} \overline{ \grad \phi}   ).
\end{equation}
for any $\xi \in T_x \mathcal{M}_p$ and its horizontal lift $\bar{\xi} \in \mathcal{H}_{\bar{x}} \mathcal{M}_p$.

\subsection*{Trust-region subproblem and retraction on $\mathcal{M}_p$}
Trust-region algorithms on a quotient manifold with guaranteed quadratic rate convergence have been proposed in \cite[Algorithm $10$]{ absil08a}. \change{The convergence of the trust-region algorithm is quadratic because the assumptions \cite[Theorem $7.4. 11$]{absil08a} are satisfied locally.} The trust-region subproblem on $\mathcal{M}$ is formulated as 
\[
\begin{array}{ll}
\min\limits_{\xi \in T_{x} \mathcal{M}_p} & \phi (x) +  g_{x} (  \xi ,   \grad \phi(x) ) + \frac{1}{2}  g_{x} (  \xi, \hess \phi(x) [\xi]    ) \\
\subject & {g}_{x}  ( \xi , \xi ) \leq \delta.
\end{array}
\]
where  $\delta$ is the trust-region radius and $\grad\phi$ and $\hess \phi$ are the Riemannian gradient and Hessian on $\mathcal{M}_p$. The problem is horizontally lifted to the horizontal space $\mathcal{H}_{\bar{x}} \mathcal{M}_p$ where it is solved using a \emph{truncated-conjugate gradient} method with parameters set as in \cite[Alg~2]{absil07a}. Solving the above trust-region subproblem leads to a direction $\xi$ that minimizes the quadratic model. 

To find the new iterate based on the obtained direction $\xi$, a mapping in the tangent space $T_{x}\mathcal{M}_p$ to the manifold $\mathcal{M}_p$ is required. This mapping is more generally referred to as \emph{retraction} which maps the vectors from the tangent space onto the points on the manifold, $R_{x}: T_{x} \mathcal{M}_p  \rightarrow \mathcal{M}_p $ (details in \cite[Definition~4.1.1]{absil08a}). In the present case, a retraction of interest is \cite{absil08a,meyer11a} 
\begin{equation}\label{eq:retraction}
\begin{array}{ll}
R_{\mat{U}} (\xi_{\mat U}) = \rm{uf}(\mat{U} + \xi_{\mat U}), \quad
R_{\mat{B}} (\xi_{\mat B}) = \mat{B}^{\frac{1}{2}} \rm{exp}   (    \mat{B}^{- \frac{1}{2}}  \xi_{\mat B}  \mat{B}^{- \frac{1}{2}}   )   \mat{B}^{\frac{1}{2}}\ {\rm and} \\
R_{\mat{V}} (\xi_{\mat V}) = \rm{uf}(\mat{V} + \xi_{\mat V}) \\
\end{array}
\end{equation}
where $\rm{uf}$ is a function that extracts the orthogonal factor of the polar factorization, \change{i.e., $\rm{uf}(\mat{A}) = \mat{A} (\mat{A}^T \mat{A})^{-1/2}$} and $\rm{exp}$ is the \emph{matrix exponential} operator. The retraction on the positive definite cone is the natural exponential mapping for the metric (\ref{eq:metric}) \cite{smith05a}. \change{These well-known retractions on the individual manifolds is also a valid retraction on the quotient space by virtue of \cite[Proposition $4.1.3$]{absil08a}.} 
 
\subsection*{Numerical complexity}
The numerical complexity per iteration of the proposed trust-region algorithm to solve (\ref{eq:factorized_formulation_two_var}) depends on the computational cost of the following components. 
\begin{itemize}
\item{Objective function $\bar{\phi} \longrightarrow$ problem dependent }
\item{Metric $\bar{g}$ $\longrightarrow$ $O(np^2 + mp^2 + p^3)$}
\item{Euclidean gradient of $\bar{\phi}$ $\longrightarrow$ problem dependent }
\item{ $ \overline{\rc}_{\bar{\nu}} \bar{\eta} =  \Psi (\D \bar{\eta} [\bar{\nu}]) - \Psi ( \nu_{\mat{U}}\Sym(\mat{U}^T \eta_{\mat{U}}),   \Sym(\nu_{\mat{B}} \mat{B}^{-1} \eta_{\mat{B}}) ,\nu_{\mat{V}}\Sym(\mat{V}^T \eta_{\mat{V}}) ) $
\begin{itemize}
\item{$\D \bar{\eta} [\bar{\nu}] \longrightarrow$ problem dependent}
\item{Matrix $\nu_{\mat{U}}\Sym(\mat{U}^T \eta_{\mat{U}})$ $\longrightarrow$ $O(np^2)$}
\item{ Matrix $\Sym(\nu_{\mat{B}} \mat{B}^{-1} \eta_{\mat{B}})$  $\longrightarrow$ $O(p^3)$ }
\item{Matrix $\nu_{\mat{V}}\Sym(\mat{V}^T \eta_{\mat{V}})$ $\longrightarrow$ $O(mp^2 )$}
\end{itemize}
}
\item{Projection operator $\Psi$ $\longrightarrow$ $O(np^2 + mp^2 )$ 
}
\item{Projection operator $\Pi$ $\longrightarrow$ $O(np^2 + mp^2 + p^3)$ 
\begin{itemize}
\item{Lyapunov equation for $\mat{\Omega}$ $\longrightarrow$ $O(p^3)$}
\end{itemize}
}
\item{Retraction $R$ $\longrightarrow$ $O(np^2 + mp^2 + p^3)$}
\end{itemize}
As shown above all the manifold related operations \change{have} linear complexity in $n$ and $m$. Other operations depend on the problem at hand and are computed in the search space $\overline{\mathcal{M}}_p$. With $p \ll \min\{n ,m\}$ the computational burden on the algorithm considerably reduces.

\section{An optimization scheme to solve convex program (\ref{eq:general_formulation})}
Starting with rank $1$ problem, we alternate a second-order local optimization algorithm on fixed-rank manifold with a first-order rank-one update. The scheme is shown in Table \ref{tab:algorithm}. \change{The rank update ensures that the cost is decreased and the new point belongs to $\overline{\mathcal{M}}_{p+1}$.}
 \begin{proposition}\label{prop:descent_directions}
 If $\mat{X} = \mat{U}\mat{B}\mat{V}^T$ \change{is a stationary point of (\ref{eq:factorized_formulation})} then the rank-one update 
  \begin{equation}\label{eq:update}
\mat{X}_+ = \mat{X} - \beta uv^T
 \end{equation}
ensures a decrease in the objective function $f(\mat{X}) + \lambda\| \mat{X} \|_*$ provided that $\beta > 0$ is sufficiently small and the descent directions $u\in \mathbb{R}^{n}$ and $v\in \mathbb{R}^{m}$ are the dominant left and right singular vectors with singular value $\sigma_1$ of the dual variable $\mat{S} = {\rm{Grad}}_{\mat{X}}f(\mat{U}\mat{B}\mat{V}^T)$. \change{The maximum decrease in the objective function is obtained for $\beta = \frac{\sigma_1 - \lambda}{L_f}$ where $L_f$ is the Lipschitz constant such that $\| \nabla f_{\mat X}(\mat{X}) -  \nabla f_{\mat Y}(\mat{Y}) \|_F \leq L_f \| \mat{X} - \mat{Y} \|_F$ for all $\mat{X}, \mat{Y} \in \mathbb{R}^{n\times m}$.}
\end{proposition}
\begin{proof}
This is in fact a descent step as shown in \cite{cai10a, ma11a, mazumder10a} but now projected onto the rank-one dominant subspace. \change{The	 proof is given in Appendix \ref{appendix:descent_directions}.}
\end{proof}

\begin{table}
\begin{center}
{\scriptsize
\framebox[5.0in]{
\begin{minipage}[t]{4.8in}
\textbf{Algorithm to solve convex problem (\ref{eq:general_formulation})}
\begin{enumerate}
\setcounter{enumi}{-1}
\item{
\begin{itemize}
\item{Initialize $p$ to $p_0$, a guess rank}.
\item{Initialize the threshold $\epsilon$ for convergence criterion, refer to Proposition \ref{prop:convergence}.}
\item{Initialize the iterates $\mat{U}_0 \in \Stiefel{p_0}{n}$, $\mat{B}_0 \in \PD{p_0}$ and $\mat{V}_0 \in \Stiefel{p_0}{m}$.}
\end{itemize}
}
\item{
Solve the non-convex problem (\ref{eq:factorized_formulation}) in the dimension $p$ to obtain a local minimum $(\mat{U}, \mat{B},\mat{V})$.
}
\item{Compute $\sigma_1$ (the dominant singular value) of dual variable $\mat{S} = \Grad_{\mat{X}}f(\mat{U}\mat{B}\mat{V}^T)$.
\begin{itemize}
\item{If $\sigma_1 - \lambda \leq \epsilon $ (or duality gap $\leq \epsilon$) due to Proposition \ref{prop:convergence}, output $\mat{X}=\mat{U}\mat{B}\mat{V}^T$ as the solution to problem (\ref{eq:general_formulation}) and stop.} 
\item{ Else, compute the update as shown in Proposition \ref{prop:descent_directions} and compute the new point $(\mat{U}_+, \mat{B}_+, \mat{V}_+)$ as described in (\ref{eq:embedding}). Set $p = p+1$ and repeat step $1$.}
\end{itemize}
}
\end{enumerate}
\end{minipage}
}
}
\end{center}
\caption{Algorithm to solve the trace norm minimization problem of type (\ref{eq:general_formulation}).}
\label{tab:algorithm}
\end{table}

A representation of $\mat{X}_+$ on $\overline{\mathcal{M}}_{p+1}$ is obtained from the singular value decomposition of $\mat{X}_+$. Since $\mat{X}_+$ is a rank-one update of $\mat{U}\mat{B}\mat{V}^T$, the singular value decomposition can be performed efficiently \cite{brand06a}. Defining $u'$ and $v'$ such that $u' = (\mat{I} - \mat{U}\mat{U}^T)u$ and $ v' = (\mat{I} - \mat{V}\mat{V}^T)(-\beta v)$, which are the orthogonal projections of $u$ and $v$ on the complementary space of $\mat{U}$ and $\mat{V}$,  the update (\ref{eq:update}) is written as
 \[
\begin{array}{lll}
\mat{X}_+  = \mat{UBV}^T - \beta u v^T = [ \mat{U} \quad  u ] \left [
\begin{array}{lll}
\mat{B} & \mat{0} \\
\mat{0} & 1\\
\end{array}
\right ]  [ \mat{V} \quad - \beta v ]^T = [ \mat{U} \quad  \frac{u'}{\| u' \|} ] \mat{K} [ \mat{V} \quad  \frac{v'}{\| v' \|} ]^T \\

\end{array}
 \]
where 
\[
\mat{K} = 
\left [
\begin{array}{lll}
1 & \mat{U}^T u\\
\mat{0} & \| u' \|\\
\end{array}
\right ] 
\left [
\begin{array}{lll}
\mat{B} & \mat{0} \\
\mat{0} & 1\\
\end{array}
\right ] 
\left [
\begin{array}{lll}
1 & -\beta \mat{V}^T v \\
\mat{0} & 1\\
\end{array}
\right ]^T.
\]
It should be noted that $\mat{K}$ is of size $(p +1 )\times (p+1)$. If $ \mat{P}' \mat{\Sigma}' {\mat{Q}'}^T$ is the singular value decomposition of $ \mat{K} $ where $\mat{P}'$ and $\mat{Q}'$ are orthonormal matrices and $\mat{\Sigma}'$ is a diagonal matrix then the new iterate $\mat{X}_+ \in \overline{\mathcal{M}}_{p+1}$ is 
\begin{equation} \label{eq:embedding}
\begin{array} {lll}
\mat{U}_+  = [ \mat{U} \quad  \frac{u'}{\| u' \|} ] \mat{P}', \quad
\mat{B}_+ =  \mat{\Sigma}'\quad {\rm and} \quad
\mat{V}_+  = [ \mat{V} \quad  \frac{v'}{\| v' \|} ] \mat{Q}'.
\end{array}
\end{equation}
\change{To compute an \emph{Armijo-optimal} $\beta$ we perform a \emph{backtracking} line search starting from the value $\frac{\sigma_1 - \lambda}{L_f}$ where $L_f$ is the Lipschitz constant for the gradient of $f$ \cite{nesterov03a}. The justification for this value is given in Appendix \ref{appendix:descent_directions}. In many problem instances a good value of $L_f$ can be well-approximated.}

\change{There is no theoretical guarantee that the algorithm in Table \ref{tab:algorithm} stops at $p = r$ where $r$ is the optimal rank. However, convergence to the global solution is guaranteed from the fact that the algorithm alternates between fixed-rank optimization and rank updates (unconstrained projected rank-1 gradient step) and both are descent iterates. Disregarding the fixed-rank step, the algorithm reduces to a gradient algorithm for a convex problem with classical global convergence guarantees. This theoretical certificate however does not capture the convergence properties of an algorithm that empirically always converges at a rank $p \ll \min \{m, n\}$ (most often at the optimal rank)}
\change{One advantage of the scheme, in contrast to trace norm minimization algorithms proposed in \cite{cai10a, toh10a, ma11a, mazumder10a}, is that it offers a tighter control of the rank at all intermediate iterates of the scheme. It should be also be emphasized that the stopping criterion threshold of the non-convex problem (\ref{eq:factorized_formulation}) and of the convex problem (\ref{eq:general_formulation}) are chosen separately. This means that rank-increments can be made after a fixed number of iterations of the manifold optimization without waiting for the trust-region algorithm to converge to a local minimum.}


\section{Regularization path}\label{sec:predictor_corrector}
In \change{most} applications, the optimal value of $\lambda$ is unknown \cite{mazumder10a} and \change{which means that in fact problem (\ref{eq:general_formulation}) be solved} for a number of regularization parameters. In addition, even if the optimal $\lambda$ is a priori known, \change{a path of solutions corresponding to different values of $\lambda$ proves interpretability to the intermediate iterates which are now global minima for different values of $\lambda$.} \change{This motivates to compute the complete regularization path of (\ref{eq:general_formulation}) for a number of values $\lambda$, i.e., defined as} $
\mat{X}^*(\lambda_{i})= {\argmin}_{\mat{X} \in \mathbb{R}^{n \times m}} \quad f(\mat{X}) + \lambda_i \| \mat{X} \|_* $ where $\mat{X}^*(\lambda_{i})$ is the solution to the $\lambda_i$ minimization problem.

 A common approach is the \emph{warm-restart} approach where the algorithm to solve the $\lambda_{i+1}$ problem is initialized from $\mat{X}^*(\lambda_i)$ and so on \cite{mazumder10a}. However, the warm-restart approach does not use the fact that the regularization path is \change{\emph{smooth}} especially when the values of $\lambda$ are close to each other. \change{An argument towards this is given later in the paragraph}. In this section we describe a \emph{predictor-corrector} scheme that takes into account the first-order smoothness and computes the path efficiently. To compute the path we take a \emph{predictor} (estimator) step to predict the solution and then rectify the prediction by a \emph{corrector} step. This scheme has been widely used in solving differential equations and regression problems \cite{park06a}. We extend the \emph{prediction} idea to the quotient manifold $\mathcal{M}_p$. The corrector step is carried out by initializing the algorithm in Table \ref{tab:algorithm} from the predicted point. If $\mat{X}^*(\lambda_i) = \mat{U}_{i} \mat{B}_{i} {\mat{V}_{i}}^T$ is the \change{fixed-rank factorization} then the solution of the $\lambda_{i+1}$ optimization problem is predicted (or estimated), i.e., $\hat{\mat{X}}(\lambda_{i+1}) = \hat{\mat{U}}_{i+1} \hat{\mat{B}}_{i+1} \hat{\mat{V}}_{i+1}^T$, by the two previous solutions $\mat{X}^*(\lambda_i) $ and $\mat{X}^*(\lambda_{i-1})$ at $\lambda_{i}$ and $\lambda_{i-1}$ respectively belonging to the same rank manifold $\mathcal{M}_p$. When $\mat{X}^*(\lambda_{i-1})$ and $\mat{X}^*(\lambda_i)$ belong to different rank manifolds we perform instead a warm restart to solve $\lambda_{i+1}$ problem. The complete scheme is shown in Table \ref{tab:predictor_corrector} and has the following advantages.
\begin{itemize}
\item{With a few number of rank increments we traverse the entire path.}
\item{Potentially every iterate of the optimization scheme is now a global solution for a value of $\lambda$.}
\item{The predictor-corrector approach outperforms the warm-restart approach in maximizing \emph{prediction accuracy} with minimal extra computations.}
\end{itemize}
In this section, we assume that the optimization problem (\ref{eq:general_formulation}) has a unique solution for all $\lambda$. A sufficient condition is that $f$ is \emph{strictly} convex, which can be enforced by adding a small multiple of the square Frobenius norm to $f$. \change{The global solution $\mat{X}^*(\lambda) = \mat{U} \mat{B} \mat{V}^T$ is uniquely characterized by the non-linear system of equations
\[
\begin{array}{lll}
\mat{SV} &= & \lambda \mat{U},\
\mat{U}^T \mat{SV} =  \lambda \mat{I} \ \  {\rm and}\   \
\mat{S}^T \mat{U} =  \lambda \mat{V} \\
\end{array}
\] 
which is obtained from the optimality conditions (\ref{eq:foc_factorized_formulation}) and Proposition \ref{prop:convergence}. The smoothness of $\mat{X}^*(\lambda)$ with respect to $\lambda$ follows from the Implicit Function Theorem \cite{krantz02a}.} \change{Another reasoning is by looking at the geometry of the dual formulation. Note that we employ the predictor-corrector step only when we are on the fixed-rank manifold which corresponds to a \emph{face} of the dual operator norm set. From (\ref{eq:dual_formulation}), the dual optimal solution is obtained by projection onto the dual set Smoothness of the dual variable $\mat{M}^*(\lambda)$ with respect to $\lambda$ follows from the smoothness of the projection operator \cite{hiriart-urruty93a}. Consequently, smoothness of the primal variable $\mat{X}^*(\lambda)$ follows from the smoothness assumption of $f$.}

\begin{table}
\begin{center}
{\scriptsize
\framebox[5.0in]{
\begin{minipage}[t]{4.8in}
\textbf{Computing the regularization path}
\begin{enumerate}
\setcounter{enumi}{-1}
\item{Given $ \{ \lambda_i\}_{i=1,...,N}$ in decreasing order. Also given are the solutions $\mat{X}^*(\lambda_1)$ and $\mat{X}^*(\lambda_2)$ at $\lambda_{1}$ and $\lambda_{2}$ respectively and their low-rank factorizations.}
\item{Predictor step:  
\begin{itemize}
\item{If $\mat{X}^*(\lambda_{i-1})$ and $\mat{X}^*(\lambda_{i})$ belong to the same quotient manifold $\mathcal{M}_p$ then construct a first-order approximation of the solution path at $\lambda_i$ and estimate $\hat{\mat{X}}(\lambda_{i+1})  $ as shown in (\ref{eq:predictor}).}
\item{Else $\hat{\mat{X}}(\lambda_{i+1})  = \mat{X}^*(\lambda_{i}) $.}
\end{itemize}
}
\item{Corrector step: Using the estimated solution of the $\lambda_{i+1}-\rm{problem}$, initialize the algorithm described in Table \ref{tab:algorithm} to compute the exact solution $\mat{X}^*(\lambda_{i+1})$.}
\item{Repeat steps $1$ and $2$ for all subsequent values of $\lambda$.}
\end{enumerate}
\end{minipage}
}
}
\end{center}
\caption{Algorithm for computing the regularization path. If $N$ is the number of values of $\lambda$ and $r$ is the number of rank increments then the scheme uses $r$ warm restarts and $N - r$ predictor steps to compute the full path.}
\label{tab:predictor_corrector}
\end{table}

\subsection*{Predictor step on the quotient manifold $\mathcal{M}_p$}
\begin{figure}[ht]
\centering
\includegraphics[scale = 0.40]{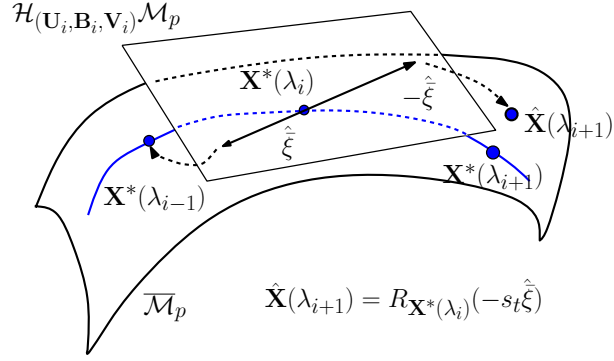}
\caption{Tracing the path of solutions using the predictor-corrector approach. The blue line denotes the curve of optimal solutions.}
\label{fig:predictor}
\end{figure}
Assuming (first-order) smoothness of the regularization path on $\mathcal{M}_p$ connecting $(\mat{U}_i, \mat{B}_i, \mat{V}_i) $ and $(\mat{U}_{i-1}, \mat{B}_{i-1}, \mat{V}_{i-1})$ in $\overline{\mathcal{M}}_p$, we build a first-order approximation of the geodesic, i.e. the curve of shortest length, connecting the two points. The estimated solution $\hat{\mat{X}}(\lambda_{i+1})$ is then computed by extending the \change{first-order approximation of the} geodesic. In other words, we need to identify a vector $\xi \in T_{[(\mat{U}_i, \mat{B}_i, \mat{V}_i)]}\mathcal{M}_p$ and its horizontal lift $\bar{{\xi}} \in \mathcal{H}_{(\mat{U}_i, \mat{B}_i, \mat{V}_i)} \mathcal{M}_p$ at $(\mat{U}_i, \mat{B}_i, \mat{V}_i)$ on $\overline{\mathcal{M}}_p$ defined as ${\bar{\xi}} = \rm{Log}_{(\mat{U}_i, \mat{B}_i, \mat{V}_i)}  (\mat{U}_{i-1}, \mat{B}_{i-1}, \mat{V}_{i-1}) $ that maps $(\mat{U}_{i-1}, \mat{B}_{i-1}, \mat{V}_{i-1})$ on $\overline{\mathcal{M}}_p$ to the horizontal space $\mathcal{H}_{(\mat{U}_j, \mat{B}_j, \mat{V}_j)} \mathcal{M}_p$ \cite{absil08a}. $\rm{Log}$ is referred to as \emph{logarithmic mapping}. Computing the logarithmic mapping (and hence, the geodesic) might be numerically costly in general. For the case of interest there is no analytic expression for the logarithmic mapping. Instead a numerically efficient way is to use the \change{approximate} inverse retraction $\hat{R}^{-1}_{(\mat{U}_i, \mat{B}_i, \mat{V}_i)} (\mat{U}_{i-1}, \mat{B}_{i-1}, \mat{V}_{i-1})$ where $\hat{R}^{-1}: \overline{\mathcal{M}}_p \rightarrow \mathcal{E}$ to obtain a direction in the space $\mathcal{E}$ followed by \emph{projection} onto the horizontal space $\mathcal{H}_{(\mat{U}_j, \mat{B}_j, \mat{V}_j)} \mathcal{M}_p$. Note that $\mathcal{E}:=  \mathbb{R}^{n \times p} \times \mathbb{R}^{p \times p} \times \mathbb{R}^{m \times p}$. The projection is accomplished using projection operators $\Psi: \mathcal{E} \rightarrow T_{(\mat{U}_i, \mat{B}_i, \mat{V}_i)} \overline{\mathcal{M}}_p$ and $\Pi: T_{(\mat{U}_i, \mat{B}_i, \mat{V}_i)} \overline{\mathcal{M}}_p \rightarrow \mathcal{H}_{(\mat{U}_i, \mat{B}_i, \mat{V}_i)} \mathcal{M}_p$ defined in Section \ref{sec:manifold_optimization}. Hence, an estimate on $\bar{\xi}$ is given as
\begin{equation}\label{eq:approximate_log}
\hat{\bar{\xi}}= \Pi (\Psi  (  \hat{R}^{-1}_{(\mat{U}_i, \mat{B}_i, \mat{V}_i)}  (\mat{U}_{i-1}, \mat{B}_{i-1}, \mat{V}_{i-1}) ) )
\end{equation}
For the retraction of interest (\ref{eq:retraction}) the \emph{Frobenius norm} error in the approximation of the Logarithmic mapping is bounded as 
\[
\begin{array}{lll}
\| \hat{\bar{\xi}} - \bar{\xi} \|_F &=  \| \hat{\bar{\xi}} -\hat{R}^{-1}_{(\mat{U}_i, \mat{B}_i, \mat{V}_i)}  (\mat{U}_{i-1}, \mat{B}_{i-1}, \mat{V}_{i-1}) + \hat{R}^{-1}_{(\mat{U}_i, \mat{B}_i, \mat{V}_i)}  (\mat{U}_{i-1}, \mat{B}_{i-1}, \mat{V}_{i-1}) - \bar{\xi} \|_F\\
& \leq \|  \hat{\bar{\xi}} -\hat{R}^{-1}_{(\mat{U}_i, \mat{B}_i, \mat{V}_i)}  (\mat{U}_{i-1}, \mat{B}_{i-1}, \mat{V}_{i-1})\|_F + \|  \hat{R}^{-1}_{(\mat{U}_i, \mat{B}_i, \mat{V}_i)}  (\mat{U}_{i-1}, \mat{B}_{i-1}, \mat{V}_{i-1}) - \bar{\xi}\|_F \\
 & \leq \min\limits_{\bar{\zeta} \in \mathcal{H}_{(\mat{U}_i, \mat{B}_i, \mat{V}_i)} \mathcal{M}_p} \| \bar{\zeta} -  \hat{R}^{-1}_{(\mat{U}_i, \mat{B}_i, \mat{V}_i)}  (\mat{U}_{i-1}, \mat{B}_{i-1}, \mat{V}_{i-1})\|_F + O(\| \bar{\xi}\|_F ^2), \\
 &\quad  {\rm as } \quad \| \bar{\xi}\| \rightarrow 0.
\end{array}
\]
The $O(\| \bar{\xi}\|_F ^2)$ approximation error comes from the fact that the retraction $R$ used is \change{at least} a first-order retraction \cite{absil08a}. This approximation is exact if $\overline{\mathcal{M}}_p$ is the Euclidean space. The \change{approximate} inverse retraction $\hat{R}^{-1}$ corresponding to the retraction $R$ described in (\ref{eq:retraction}) is computed as
\[
\begin{array}{ll}
\hat{R}^{-1}_{\mat{U}_i}( \mat{U}_{i-1}) = \mat{U}_{i-1} - \mat{U}_{i}, \ \ 
\hat{R}^{-1}_{\mat{B}_i}( \mat{B}_{i-1})  =  \mat{B}_i^{\frac{1}{2}} \rm{log}   (    \mat{B}_i^{- \frac{1}{2}}  {\mat B}_{i-1}  \mat{B}_i^{- \frac{1}{2}}  )   \mat{B}_i^{\frac{1}{2}} \\  
\hat{R}^{-1}_{\mat{V}_i}( \mat{V}_{i-1}) = \mat{V}_{i-1} - \mat{V}_{i} \\
\end{array}
\]
where $\rm{log}$ is the matrix logarithm operator. The predicted solution is then obtained by taking a step $s_t$ and performing a backtracking line search in the direction $-\hat{\bar{\xi}}$ i.e.,
\begin{equation}\label{eq:predictor}
 ( \hat{\mat{U}}_{i+1}, \hat{\mat{B}}_{i+1}, \hat{\mat{V}}_{i+1}  ) =  R_{(\mat{U}_i, \mat{B}_i, \mat{V}_i)} (-s_t \hat{\bar{\xi}}).
\end{equation}
A good choice of the initial step size $s_t$ is $ \frac{\lambda_{j+1} - \lambda_j}{ \lambda_j - \lambda_{j-1}}$. The motivation for the choice comes the observation that it is optimal when the solution path is a straight line in the Euclidean space. The numerical complexity to perform the prediction step in the manifold $\overline{\mathcal{M}}_p$ is $O(np^2 + mp^2 + p^3)$.

   
\section{Numerical Experiments}\label{sec:numerical_experiments}
The overall optimization scheme with \emph{descent-restart} and trust-region algorithm is denoted as ``Descent-restart + TR'' (TR).  We test the proposed optimization framework on the problems of low-rank matrix completion and multivariate linear regression where trace norm penalization has shown efficient recovery. Full regularization paths are constructed with optimality certificates. All simulations in this section are performed in MATLAB on a $2.53$ GHz Intel Core $\rm{i}5$ machine with $4$ GB of RAM.

\subsection{Diagonal versus matrix scaling}\label{sec:polar_vs_svd}
Before entering a detailed numerical experiment we illustrate here the empirical evidence that constraining $\mat{B}$ to be diagonal (as is the case with SVD) is detrimental to optimization. To this end, we consider the simplest implementation of a gradient descent algorithm for matrix completion problem (see below). The plots shown Figure \ref{fig:polar_vs_svd} compare the behavior of the same algorithm in the search space $\Stiefel{p}{n} \times \PD{p} \times \Stiefel{p}{m}$ and $\Stiefel{p}{n} \times {\rm Diag}_+(p) \times \Stiefel{p}{m}$ (SVD). ${\rm Diag}_+(p)$ is the set of diagonal matrices with positive entries. The empirical observation that convergence suffers from imposing diagonalization on $\mat{B}$ is a generic observation that \change{does not} depend on the particular problem at hand. The problem here involves completing a $200\times 200$ of rank $5$ from $40\%$ of observed entries. $\lambda$ is fixed at $10^{-10}$.
\begin{figure}[ht]
\centering
\includegraphics[scale = 0.33]{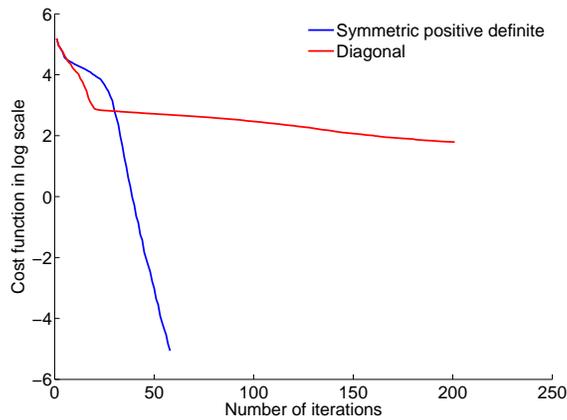}
\caption{Convergence of a gradient descent algorithm is affected by making $\mat{B}$ diagonal.}
\label{fig:polar_vs_svd}
\end{figure}

\subsection{Low-rank matrix completion}
The problem of matrix completion involves completing an $n\times m$ matrix when only a few entries of the matrix entries are known. Presented in this way the problem is ``ill-posed'' but becomes considerably interesting when in addition a \emph{low-rank reconstruction} is also sought. Given an incomplete low-rank (but unknown) $n\times m$ real matrix $\widetilde{\mat{X}}$, a convex relaxation of the matrix completion problem is 
\begin{equation}\label{eq:matrix_completion} 
\begin{array}{ll}
\min\limits_{\mat{X} \in \mathbb{R}^{n \times m}} & \| \mat{W}\odot (  \widetilde{\mat{X}}  -  \mat{X}   ) \|_F ^2 + \lambda\|\mat{X}\|_*\\
\end{array}
\end{equation}
for $\mat{X} \in \mathbb{R}^{n \times m}$ and a regularization parameter $\lambda$. Here $\| \cdot \|_F$ denotes the Frobenius norm, matrix $\mat{W}$ is an $n \times m$ \emph{weight} matrix with binary entries and the operator $\odot$ denotes element-wise multiplication. If $\mathcal{W}$ is the set of known entries in $\widetilde{\mat{X}}$ then, $\mat{W}_{ij} = 1$ if $(i,j) \in \mathcal{W}$ and $\mat{W}_{ij} = 0$ otherwise. The problem of matrix completion is known to be combinatorially hard. However, by solving the convex relaxation (\ref{eq:matrix_completion}) a low-rank reconstruction is possible with a very high probability \cite{candes09a,  keshavan09b} under certain assumptions on the number of observed entries. For an exact reconstruction, the lower bound on the number of known entries is typically of the order $O(nr + mr)$ where $r$ is the rank of the optimal solution, \change{$| \mathcal{W}| >  {\rm max} \{n, m \} \gg r$}. Consequently, it leads to a very sparse weight matrix $\mat{W}$, which plays a very crucial role for efficient algorithmic implementations. For our case, we assume that the lower bound on the number of entries is met and we seek a solution to the optimization problem (\ref{eq:matrix_completion}). Customizing the terminology for the present problem, the convex function is \change{$f(\mat{X}) = \| \mat{W}\odot (  \widetilde{\mat{X}}  -  \mat{X}   ) \|_F ^2$.} Using the factorization $\mat{X} = \mat{UBV}^T$, the rank-$p$ \change{objective function} is $\bar{\phi}(\mat{U}, \mat{B} ,\mat{V}) = \| \mat{W}\odot (  \widetilde{\mat{X}}  -  \mat{UBV}^T   ) \|_F ^2 + \lambda \trace(\mat{B})$ where $(\mat{U} ,\mat{B},\mat{V}) \in \overline{\mathcal{M}}_p$. The dual variable $\mat{S} = 2( \mat{W} \odot (\mat{UBV}^T - \widetilde{\mat{X}}))$.

The matrix representation of the gradient of $\bar{\phi}$ in \change{$\mathcal{E} := \mathbb{R}^{n\times r} \times \mathbb{R}^{r\times r} \times \mathbb{R}^{m\times r}$} is $\Grad_{\mat{U}} \bar{\phi} =   \mat{S}  \mat{VB}$, $\Grad_{\mat{B}} \bar{\phi} = \mat{U}^T \mat{S} \mat{V} + \lambda \mat{I}$ and $\Grad_{\mat{V}} \bar{\phi}  =  \mat{S}^T \mat{UB}$. The Euclidean directional derivative of the gradient of $\bar{\phi}$ along $\mat{Z} = (\mat{Z}_{\mat{U}},\mat{Z}_{\mat{B}}, \mat{Z}_{\mat{V}}) \in T_{\bar x}\overline{\mathcal M}_p$ is $(\mat{SV} \mat{Z}_{\mat{B}} + \mat{S}\mat{Z}_{\mat{V}}\mat{B} + \mat{S}_* \mat{VB},
     \ \ \mat{Z}_{\mat{U}}^T \mat{S} \mat{V} + \mat{U}\mat{S}\mat{Z}_{\mat{V}}  + \mat{U}^T \mat{S}_*  \mat{V}         ,\ \
      \mat{S}^T \mat{U}\mat{Z}_{\mat{B}} + \mat{S}^T\mat{Z}_{\mat{U}}\mat{B} + \mat{S}_*^T \mat{UB})
$ where $\mat{S}_* = \D_{(\mat{U}, \mat{B}, \mat{V})} \mat{S} [\mat{Z}] =  2( \mat{W} \odot (\mat{Z}_{\mat{U}}\mat{B}\mat{V}^T +  \mat{U} \mat{Z}_{\mat{B}} \mat{V}^T  + \mat{UB} \mat{Z}_{\mat{V}}^T   ))$ is the directional derivative of $\mat{S}$ along $\mat{Z}$. The Riemannian gradient and Hessian are computed using formulae developed in \change{(\ref{eq:riemannian_gradient}) and (\ref{eq:riemannian_hessian})}. Note that since $\mat{W}$ is sparse, $\mat{S}$ and $\mat{S}_*$ are sparse too. \change{As a consequence, the numerical complexity per iteration for the trust-region algorithm is of order $O(|\mathcal{W}|p + np^2 + mp^2 + p ^3)$ where $|\mathcal{W} |$ is the number of known entries. In addition computation of dominant singular value and vectors is performed with numerical complexity of $O(|\mathcal{W}|)$ \change{\cite{propack98a}}}. The \change{overall} linear complexity with respect to the number of known entries allows us to handle potentially very large datasets.

\subsection*{Fenchel dual and duality gap for matrix completion}
For the matrix completion problem, the sampling operation is the linear operator $\mathcal{A}(\mat{X}) = \mat{W} \odot \mat{X}$. We can, therefore, define a new function $\psi$ such that $f(\mat{X}) = \psi ( \mat{W} \odot \mat{X})$. \change{The domain of $\psi$ is the non-zero support of $\mat{W}$.} The dual candidate $\mat{M}$ is defined by
$
\mat{M} =  \min(1, \frac{\lambda }{\sigma_\psi}) \Grad \psi
$
 where $\Grad \psi (\mat{W} \odot \mat{X}) = 2 (\mat{W} \odot \mat{X} - \mat{W} \odot \widetilde{\mat{X}})$ and $\sigma_\psi$ is the dominant singular value of $\mathcal{A}^*(\Grad \psi)$ (refer Section \ref{sec:duality_gap} for details). In matrix form, $\mathcal{A}^*(\Grad \psi)$ \change{is written as} $\mat{W}\odot \Grad \psi$. Finally, the Fenchel dual $\psi^*$ at a dual candidate $\mat{M}$ can be computed is $\psi^*(\mat{M})  = \frac{\trace(\mat{M}^ T \mat{M})}{4} + \trace(\mat{M}^T (\mat{W} \odot \widetilde{\mat{X}}))$. The final expression for the duality gap at a point $\mat{X}$ and a dual candidate $\mat{M} =  \min(1, \frac{\lambda }{\sigma_\psi}) \Grad \psi$ is $f(\mat{X}) + \lambda \| \mat{X}\|_* + \frac{\trace(\mat{M}^ T \mat{M})}{4} + \trace(\mat{M}^T (\mat{W} \odot \widetilde{\mat{X}}))$.

\change{Next we provide some benchmark simulations for the low-rank matrix completion problem. For each example, a $ n \times m $ random matrix of rank $r$ is generated according to a Gaussian distribution with zero mean and unit standard deviation and a fraction of the entries are randomly removed with uniform probability. The dimensions of $ n \times m $ matrices of rank $r$ is $(n + m - r)r$. The over-sampling (OS) ratio determines the number of entries that are known. A over-sampling ratio of $6$ means that $6(n + m - r)r$ number of randomly and uniformly selected entries are known a priori out of $nm$ entries.}

\subsubsection{An example}\label{sec:an_example}
A $100\times 100$ random matrix of rank $10$ is generated as mentioned above. $20\%$ (${\rm OS} = 4.2$) of the entries are randomly removed with uniform probability. To reconstruct the original matrix we run the optimization scheme proposed in the Table \ref{tab:algorithm} along with the trust-region algorithm to solve the non-convex problem. For illustration purposes $\lambda$ is fixed at $1\times10^{-5}$. We also assume that we do not have any a priori knowledge of the optimal rank and, thus, start from rank $1$. The trust-region algorithm stops when the relative or absolute variation of the cost function falls below $1\times10^{-10}$. The rank-incrementing strategy stops when relative duality gap is less than $1\times10^{-5}$, i.e., $\frac{ f (\mat{X}) + \lambda \| \mat{X}\|_* + \psi^*(\mat{M})  }{|\psi^*(\mat{M})|} \leq 1\times 10^{-5}$. Convergence plots of the scheme are shown in Figure \ref{fig:matrix_completion_example}. 
\begin{figure*}[ht]
\begin{minipage}{0.5\textwidth}
\includegraphics[scale = .30]{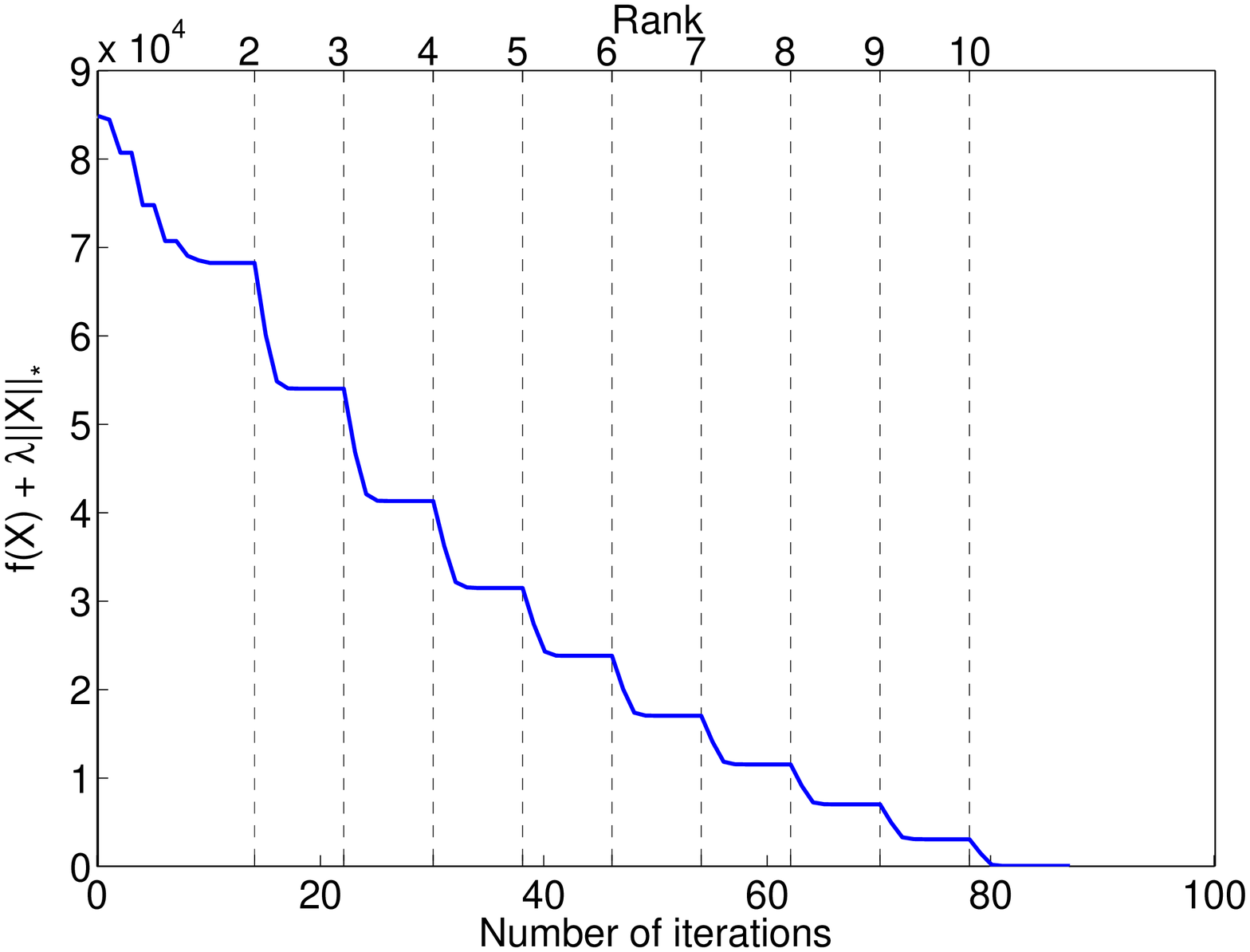}
\end{minipage}
\begin{minipage}{0.5\textwidth}
\includegraphics[scale = .30]{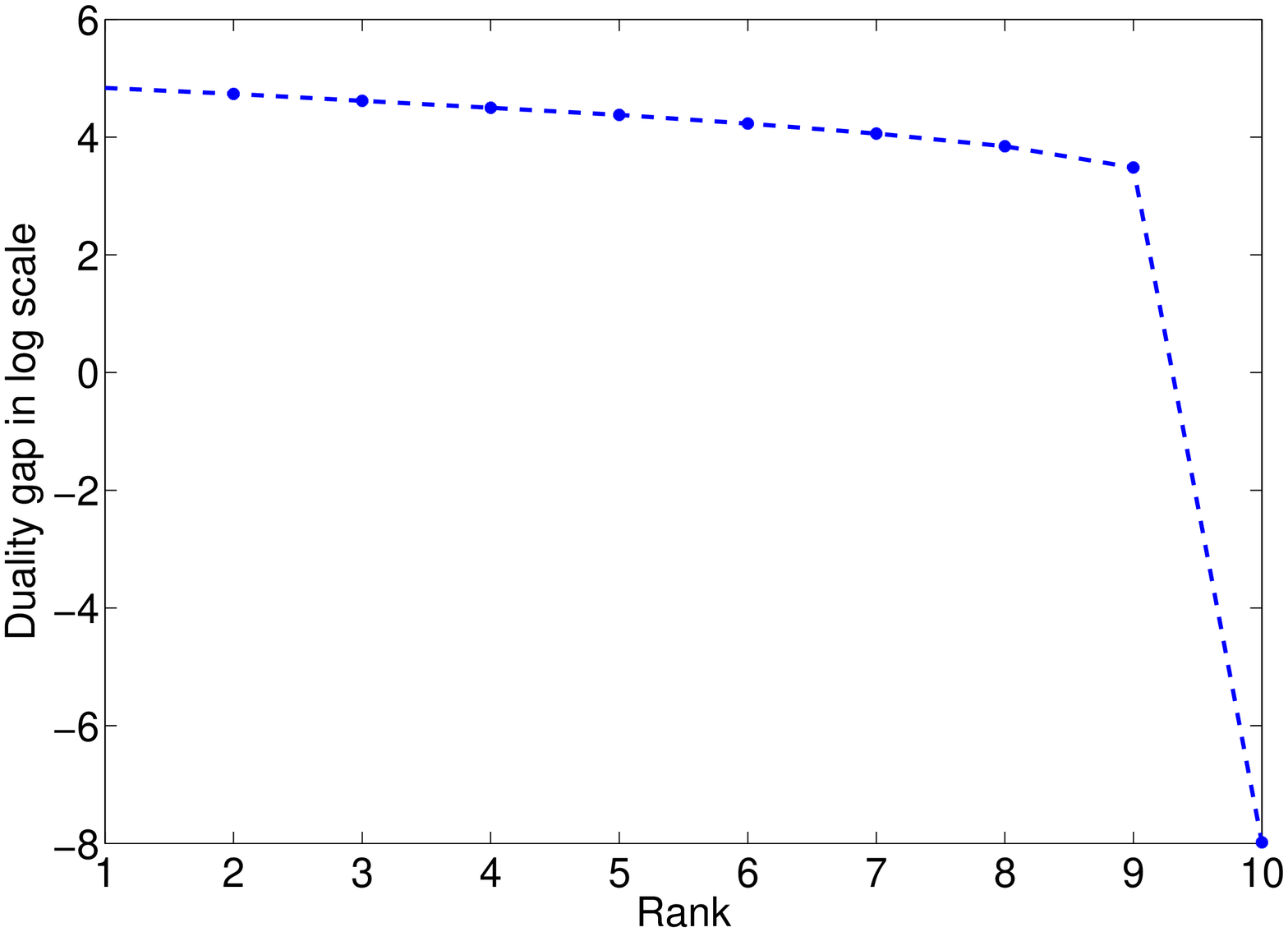}
\end{minipage}
\begin{minipage}{0.5\textwidth}
    \includegraphics[scale = .30]{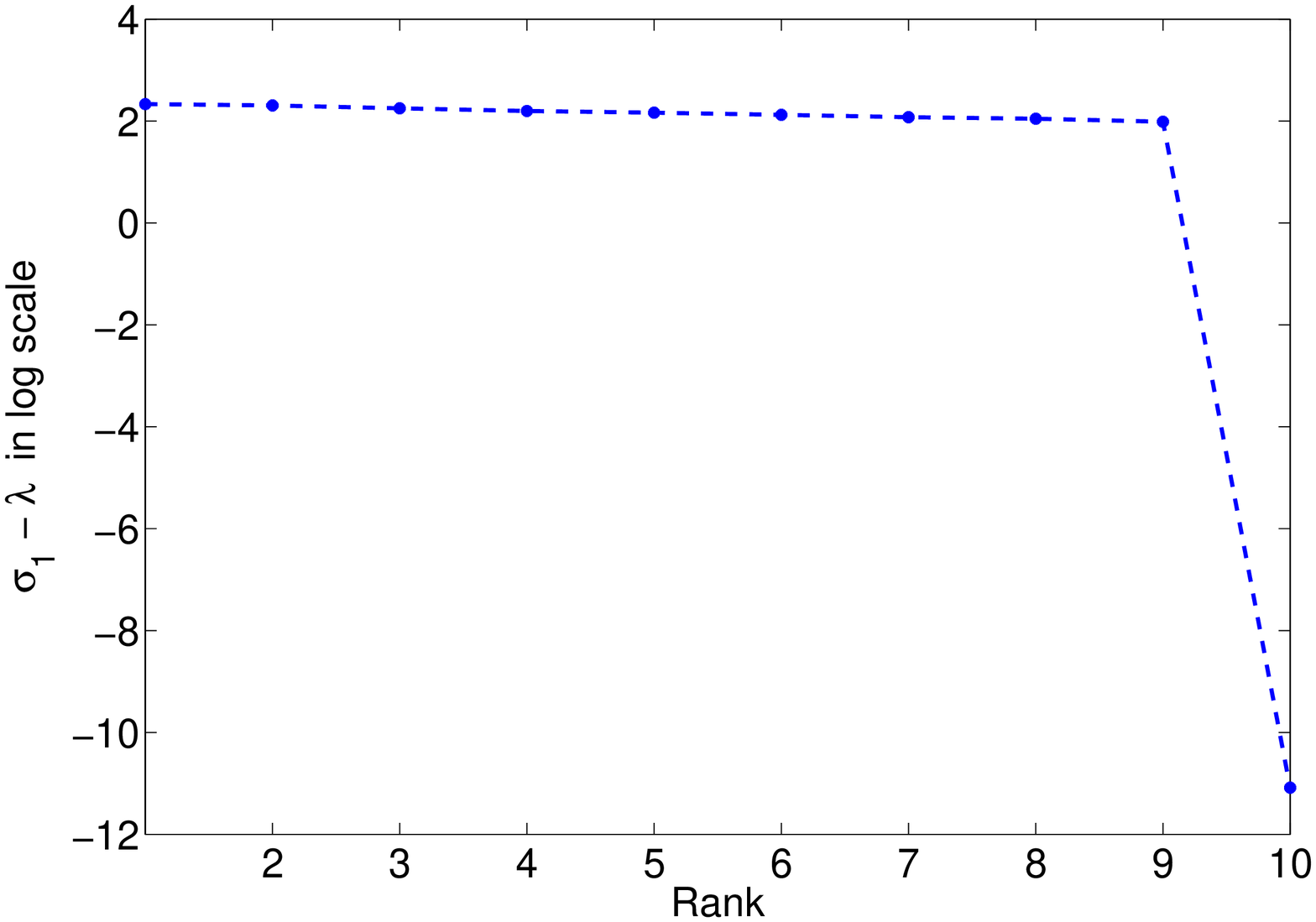}
\end{minipage}
\begin{minipage}{\textwidth}
   \scriptsize
   \begin{tabular}{|c|c|c|c|c|c|} \hline   
      Rel. error of reconstruction & $ \ $ \\
        $\| \widetilde{\mat{X}}  - \mat{X}^* \|_F / \| \widetilde{\mat{X}} \|_F  $      & $6.86 \times 10^{-8}$      \\ \hline
      Recovered rank & $10$    \\ \hline 
      Duality gap & $ 1.04 \times 10^{-8}$ \\ \hline
      $\sigma_1 - \lambda$ & $8.27 \times 10^{-12}$ \\ \hline
       Iterations & $88$  \\ \hline     
   \end{tabular}
\end{minipage} 
\caption{Matrix completion by trace norm minimization algorithm with $\lambda = 1\times 10^{-5}$. Upper left: Rank incremental strategy with descent directions. Upper right: Optimality certificate of the solution with duality gap. Lower left: Convergence to the global solution according to Proposition \ref{prop:convergence} . Lower right: Recovery of the original low-rank matrix.}
\label{fig:matrix_completion_example}
\end{figure*}
A good way to characterize matrix reconstruction at $\mat{X}$ is to look at the relative error of reconstruction, defined as,
\[
\mathrm{Rel.\  error\  of \ reconstruction } =  \| \widetilde{\mat{X}}  - \mat{X} \|_F / \| \widetilde{\mat{X}} \|_F.  
\]
Next, to understand low-rank matrix reconstruction by trace norm minimization we repeat the experiment for a number of values of $\lambda$ all initialized from the same starting point and report the relative reconstruction error in Table \ref{tab:matrix_completion_vary_lambda} averaged over $5$ runs. This, indeed, confirms that matrix reconstruction is possible by solving the trace norm minimization problem (\ref{eq:matrix_completion}).

\begin{table}
\begin{center} \scriptsize
\begin{tabular}{|c|c|c|c|c|c|c|c|c|c|c|c|c|c|c|} \hline  
$\lambda$ & $10$ & $10^{-2}$ & $10^{-5}$ & $10^{-8}$  \\ \hline
Rel. reconstruction error&  $6.33 \times  10^{-2}$   & $7.42 \times  10^{-5}$  &  $7.11 \times  10^{-8}$  & $6.89 \times  10^{-11}$  \\ \hline
Recovered rank & $10$  &  $10$ & $10$  & $10$  \\ \hline
Iterations & $113$  & $120$     & $119$ &   $123$             \\ \hline
Time in seconds & $ 2.7$  & $2.8$ &$2.9$ & $2.9$        \\ \hline
\end{tabular}
\end{center} 
\caption{Efficacy of trace norm penalization to reconstruct low-rank matrices by solving (\ref{eq:matrix_completion}). } 
\label{tab:matrix_completion_vary_lambda} 
\end{table}

\subsubsection{Regularization path for matrix completion}
In order to compute the entire regularization path, we employ the predictor-corrector approach described in Table \ref{tab:predictor_corrector} to find solutions for a grid of $\lambda$ values. For the purpose of illustration, a geometric sequence of $\lambda$ values is created with the maximum value fixed at $\lambda_{1} = 1\times10^{3}$, the minimum value is set at $\lambda_{N} = 1\times 10^{-3}$ and a reduction factor $\gamma = 0.95$ such that $\lambda_{i+1} = \gamma \lambda_i$. \change{We consider the same example as in Section \ref{sec:an_example}.} The algorithm for a $\lambda_{i} \in \{ \lambda_{1}, ... , \lambda_{N}\}$ stops when the relative duality gap falls below $1\times 10^{-5}$. Various plots are shown in Figure \ref{fig:matrix_completion_regpath_example}. Figure \ref{fig:matrix_completion_regpath_example} also demonstrates the advantage of the scheme in Table \ref{tab:predictor_corrector} with respect to a warm-restart approach. We compare both approaches on the basis of
\begin{equation}\label{eq:prediction}
{\rm{Inaccuracy\  in\  prediction}} = \bar{\phi}(\hat{\mat{X}}(\lambda_i)) -  \bar{\phi}(\mat{X}^*(\lambda_i))
\end{equation}
where $\mat{X}^*(\lambda_i)$ is the global minimum at $\lambda_i$ and $\hat{\mat{X}}(\lambda_i)$ is the prediction. A lower inaccuracy means better prediction. 
\begin{figure*}[ht]
\begin{minipage}{0.5\textwidth}
\includegraphics[scale = .30]{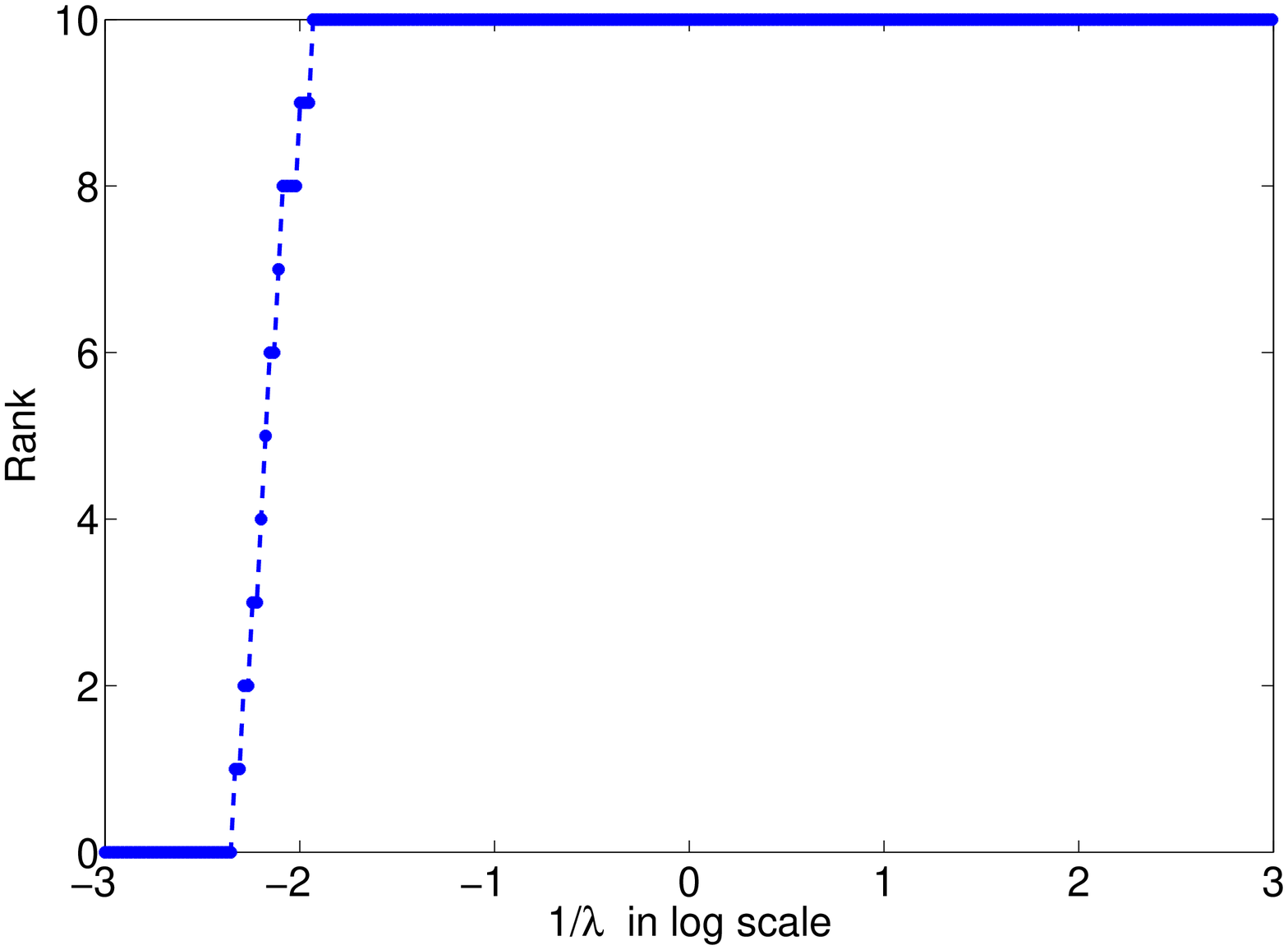}
\end{minipage}
\begin{minipage}{0.5\textwidth}
\includegraphics[scale = .30]{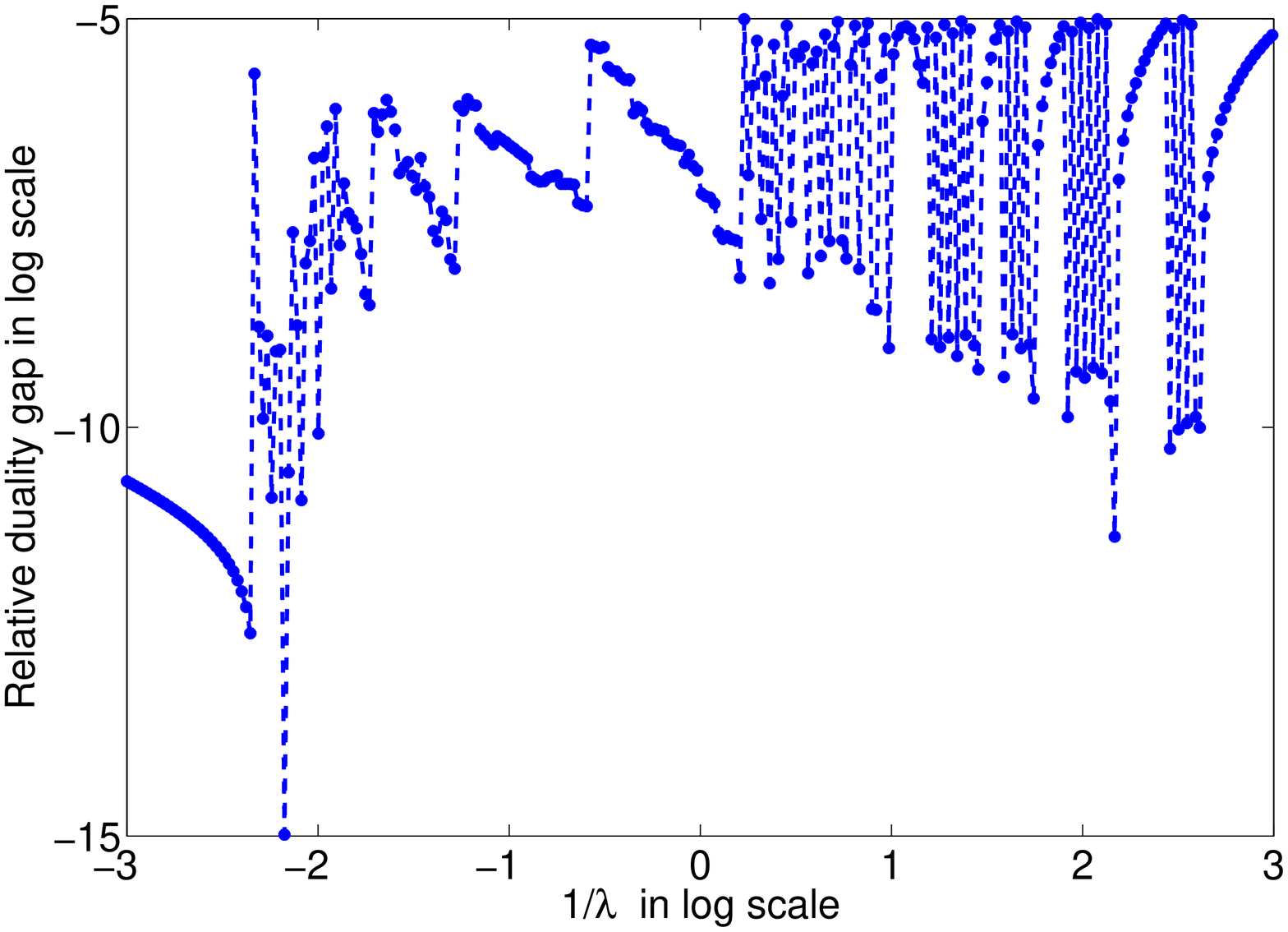}
\end{minipage}
\begin{minipage}{0.5\textwidth}
    \includegraphics[scale = .30]{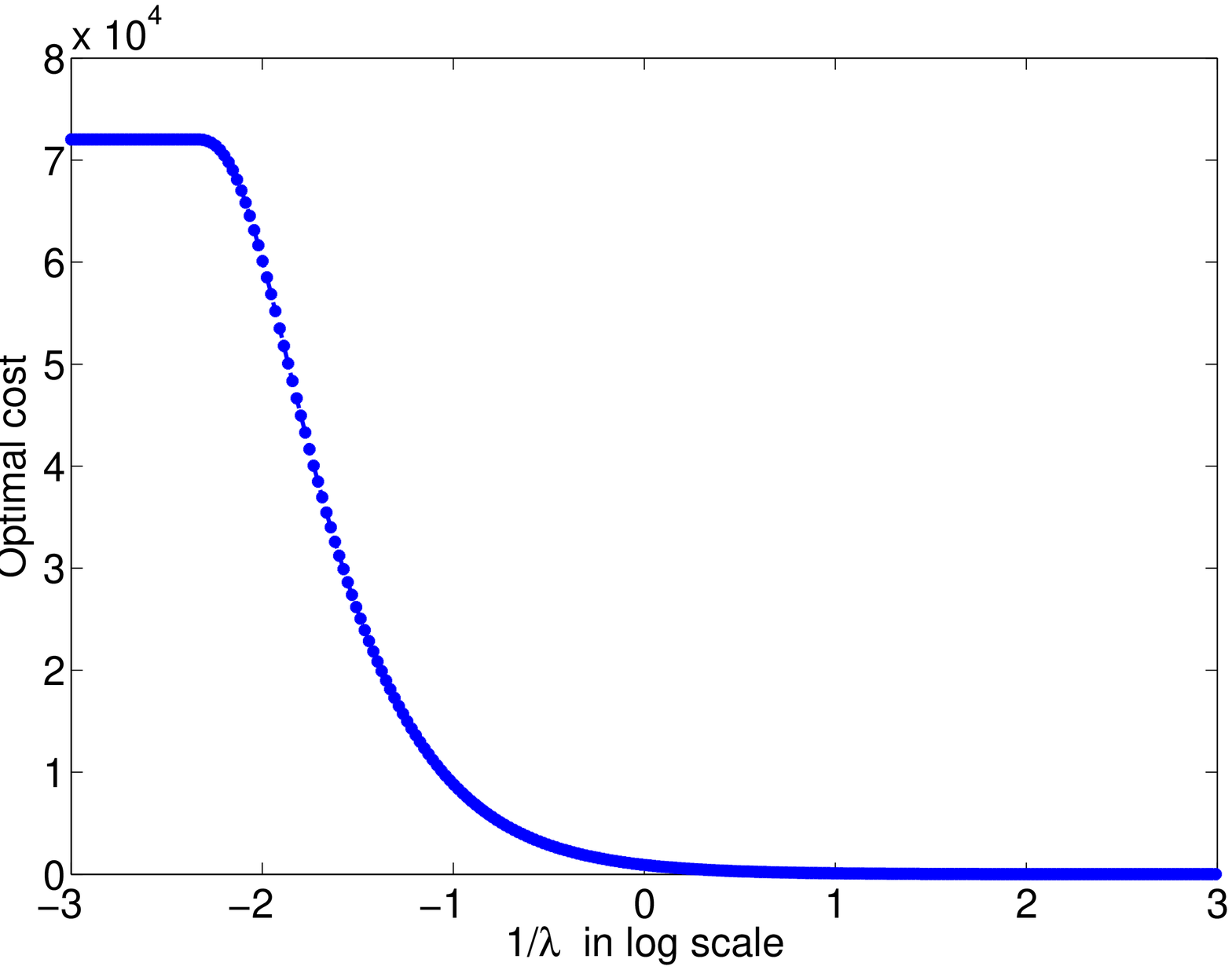}
\end{minipage}
\begin{minipage}{0.5\textwidth}
    \includegraphics[scale = .30]{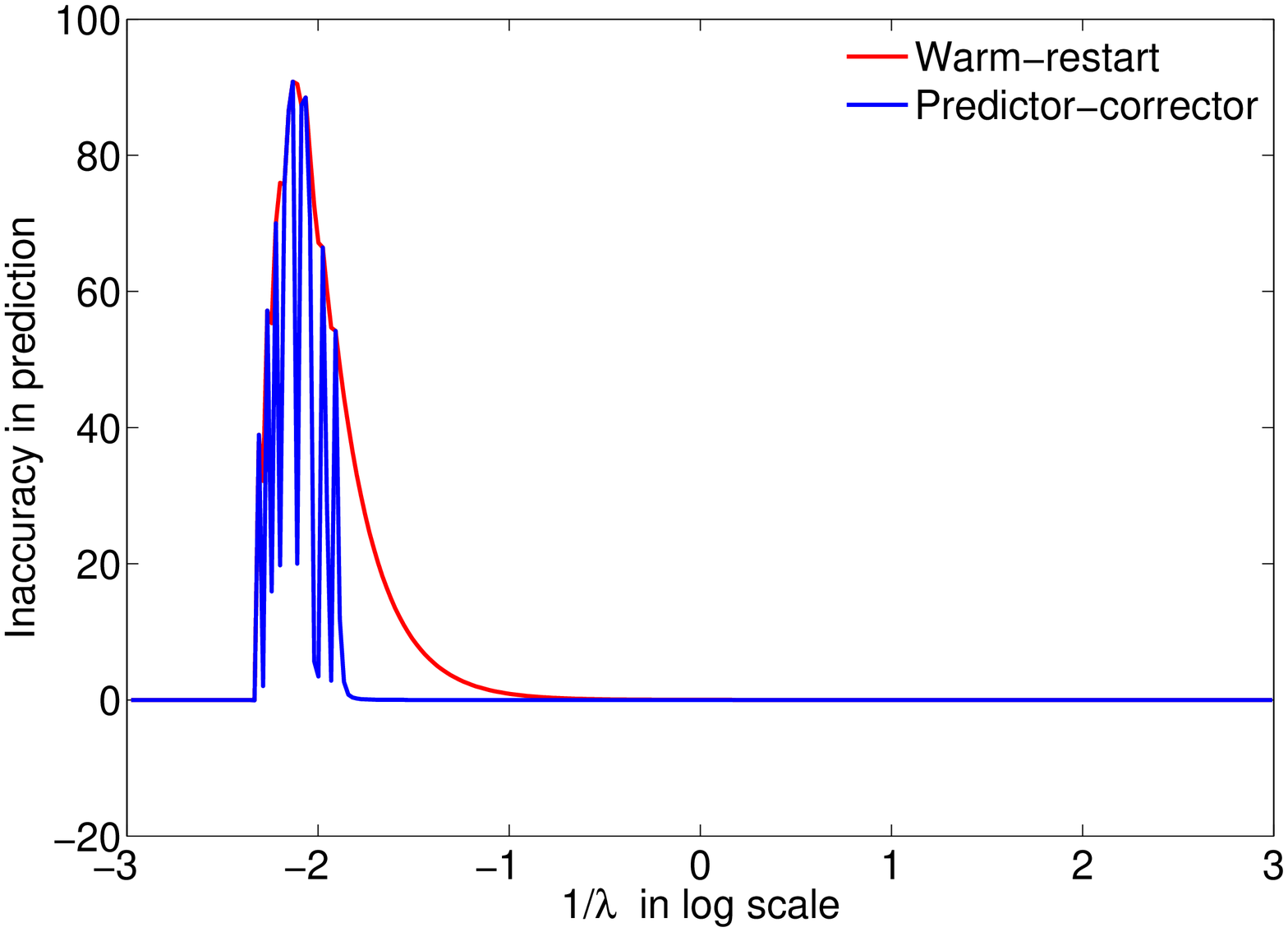}
\end{minipage}
\center
\begin{minipage}{0.25\textwidth}  
\scriptsize
   \begin{tabular}{|c|c|c|c|c|c|} \hline   
       \# $\lambda$ values & $270$      \\ \hline
      \# iterations & $766$    \\ \hline 
      Time & $38.60$ seconds \\ \hline     
   \end{tabular}
\end{minipage}
\caption{Computation of entire regularization path using Descent-restart + TR with a predictor-corrector approach. Upper left: Recovery of solutions of all ranks. Upper right: Optimality certificate for the regularization path. Lower left: Path traced by the algorithm. Lower right: Better prediction by the algorithm in Table \ref{tab:predictor_corrector} than a pure warm-restart approach. Table: Number of iterations per value of $\lambda$ is $ < 3$.}
\label{fig:matrix_completion_regpath_example}
\end{figure*}
It should be emphasized that in Figure \ref{fig:matrix_completion_example} most of the points on the curve of the objective function have no other utility than being intermediate iterates towards the global solution of the algorithm. In contrast all the points of the curve of optimal cost values in  Figure \ref{fig:matrix_completion_regpath_example} are now global minima for different values of $\lambda$.

\subsubsection{Competing methods for matrix completion}
In this section, we analyze the following state-of-the-art algorithms for low-rank matrix completion, namely, 
\begin{enumerate}
\item{SVT algorithm by Cai et al. \cite{cai10a}}
\item{FPCA algorithm by Ma et al. \cite{ma11a}}
\item{SOFT-IMPUTE (Soft-I) algorithm by Mazumder et al. \cite{mazumder10a}}
\item{\change{ APG and APGL algorithms by Toh et al. \cite{toh10a}}}
\end{enumerate}
While FPCA, SOFT-IMPUTE and APGL solve (\ref{eq:matrix_completion}), the iterates of SVT converge towards a solution of the optimization problem
\[
\begin{array}{ll}
\min\limits_{\mat X} & \tau \|\mat{X} \|_* +\frac{1}{2} \|\mat{X} \|_F^2\\
\subject & \mat{W}\odot \mat{X} = \mat{W} \odot \widetilde{\mat{X}}
\end{array}
\]
where $\tau$ is a regularization parameter. For simulation studies we use the MATLAB codes supplied on the authors' webpages for SVT, FPCA and APGL. Due to simplicity of the SOFT-IMPUTE algorithm we use our own MATLAB implementation. The numerically expensive step in all these algorithms is the computation of the \emph{singular value thresholding} operation. To reduce the computational burden FPCA uses a linear time approximate singular value decomposition (SVD). \change{Likewise, implementations of SVT, SOFT-IMPUTE and APGL exploit the low-rank + sparse structure of the iterates to optimize the thresholding operation \cite{propack98a}.} 

\change{The basic algorithm FPCA by Ma et al. \cite{ma11a} is a fixed-point algorithm with a proven bound on the iterations for $\epsilon-$accuracy. To accelerate the convergence they use the technique of \emph{continuation} that involves approximately solving a sequence of parameters leading to the target $\lambda$. The singular value thresholding burden step is carried out by a linear time approximate SVD.} \change{The basic algorithm APG of Toh et al. is a proximal method \cite{nesterov03a} and gives a much stronger bound $O(\frac{1}{\sqrt{\epsilon}})$ on the number of iterations for $\epsilon$ accuracy. To accelerate the scheme, the authors propose three additional heuristics: continuation, truncation (hard-thresholding of ranks by projecting onto fixed-rank matrices) and a line-search technique for estimating the Lipschitz constant. The accelerated version is called APGL.} \change{The basic algorithm SOFT-IMPUTE iteratively replaces the missing elements with those given by an approximate SVD thresholding at each iteration. Accelerated versions involve post processing like continuation and truncation. It should be emphasized that the performance of SOFT-IMPUTE greatly varies with the singular values computation at each iteration. For our simulations we compute $20$ dominant singular values at each iteration of SOFT-IMPUTE.}

\subsection*{Convergence behavior with varying $\mat{\lambda}$}
In this section we analyze the algorithms FPCA, SOFT-IMPUTE and Descent-restart + TR in terms of their ability to solve (\ref{eq:matrix_completion}) for a fixed value of $\lambda$. \change{For this simulation, we use FPCA, SOFT-IMPUTE and APGL without any acceleration techniques like continuation and truncation.} SVT is not used for this test since it optimizes a different cost function. We plot the objective function $f(\mat{X}) + \lambda\| \mat{X}\|_*$ against the number of iterations for a number of $\lambda$ values. \change{A $100 \times 100$ random matrix of rank $5$ is generated under standard assumptions with over-sampling ratio ${\rm OS} = 4$ ($61\%$ of entries are removed uniformly).} The algorithms Descent-restart + TR, FPCA and SOFT-IMPUTE and APG are initialized from the same point. The algorithms are stopped when either the variation or relative variation of $f(\mat{X}) + \lambda \| \mat{X}\|_*$ is less than $1\times 10^{-10}$. The maximum number of iterations is set to$500$. The rank incrementing procedure of our algorithm is stopped when the relative duality gap falls below $1\times 10^{-5}$.

%
%
%
%
%
\begin{figure*}[ht]
\begin{minipage}{0.5\textwidth}
\includegraphics[scale = 0.30]{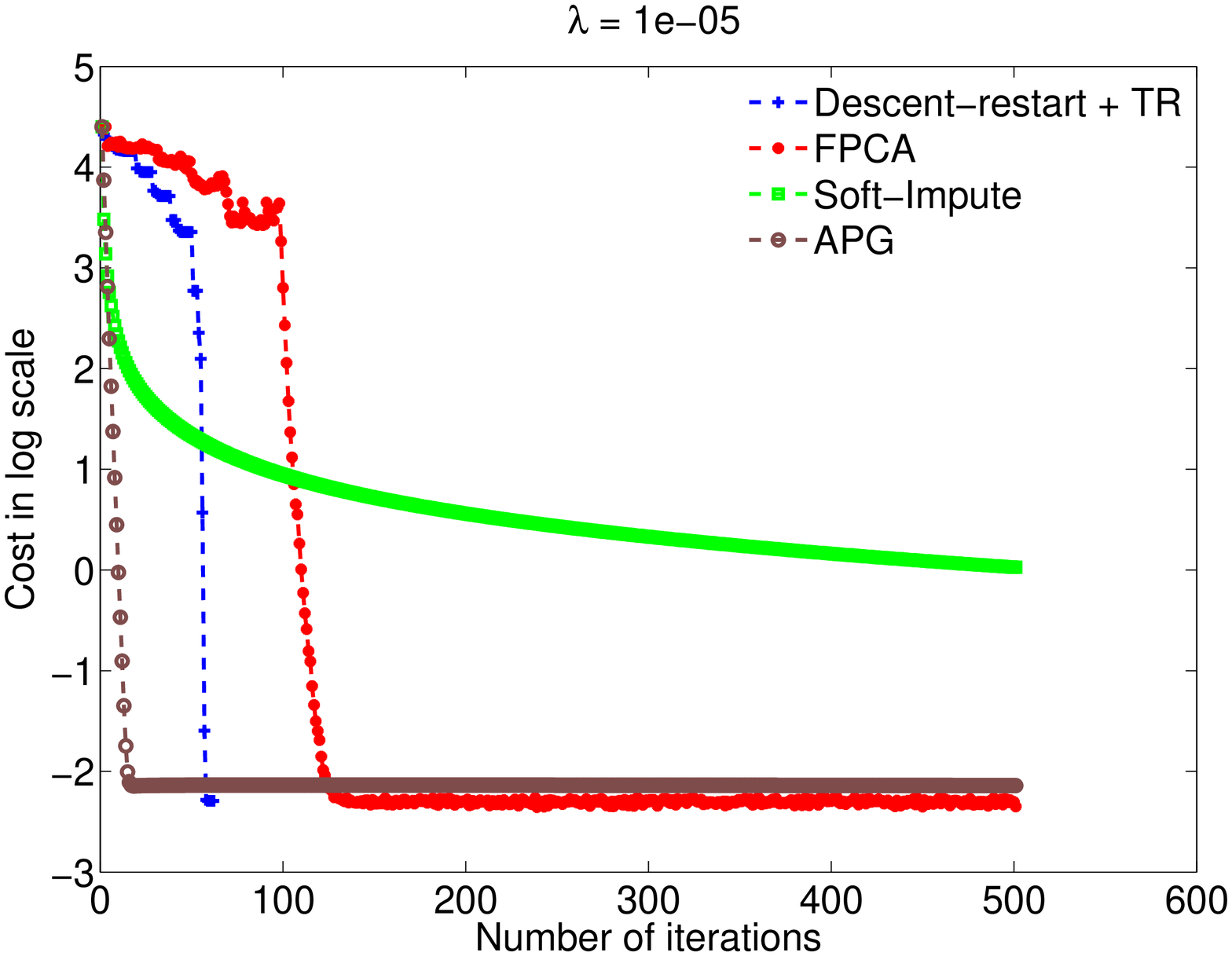}
\end{minipage}
\begin{minipage}{0.5\textwidth}
\includegraphics[scale = 0.30]{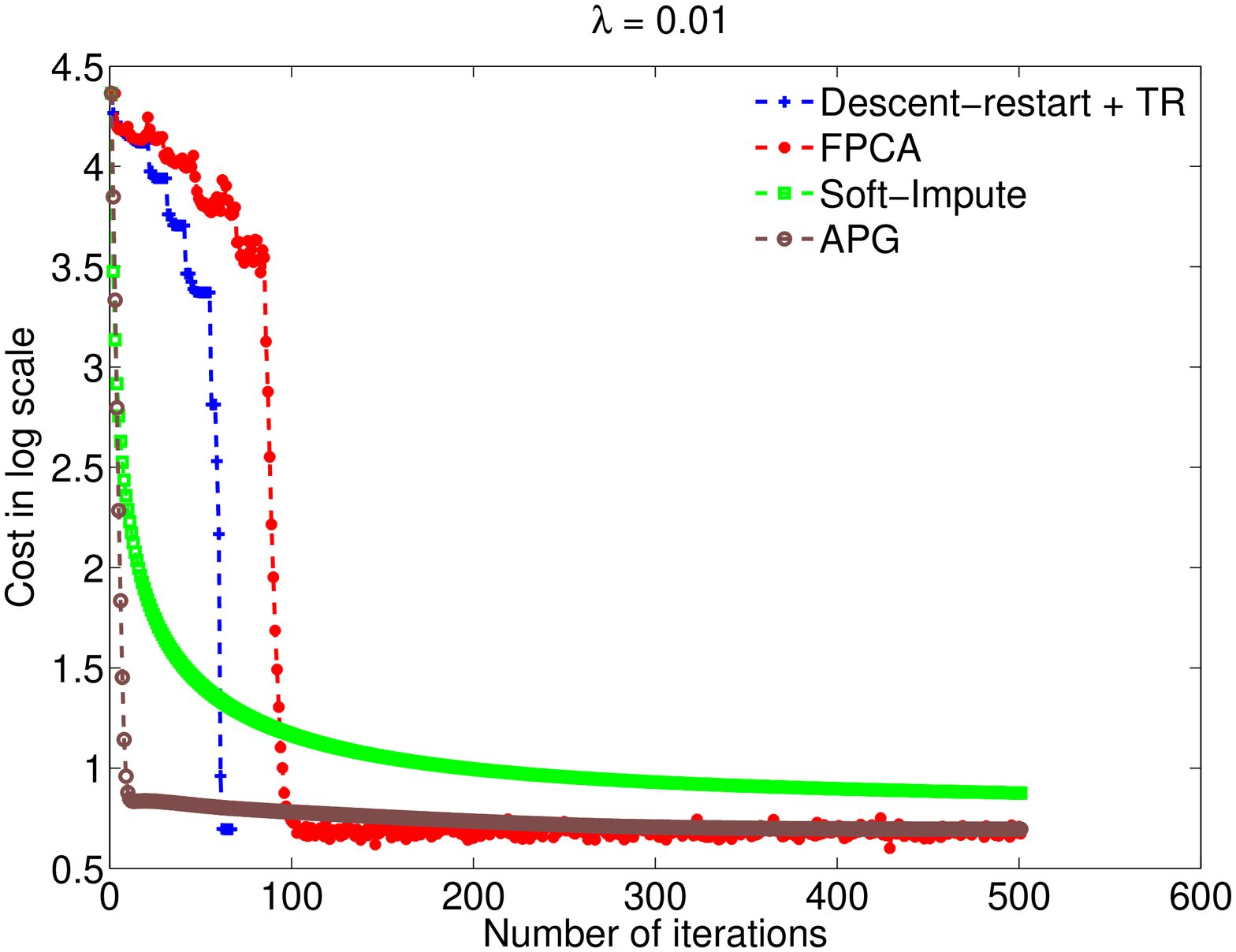}
\end{minipage}
\begin{minipage}{0.5\textwidth}
    \includegraphics[scale = 0.30]{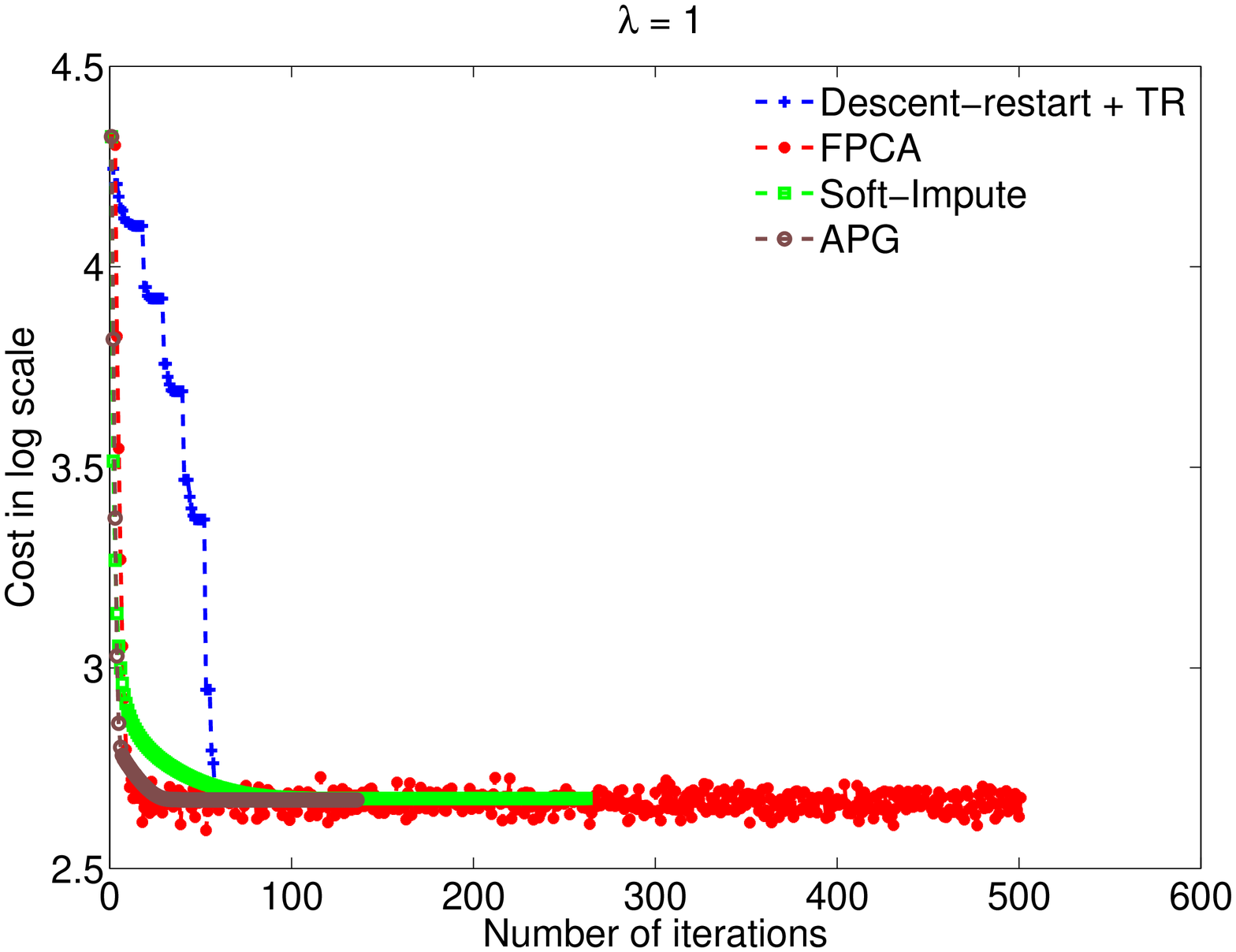}
\end{minipage}
\begin{minipage}{0.5\textwidth}
    \includegraphics[scale = 0.30]{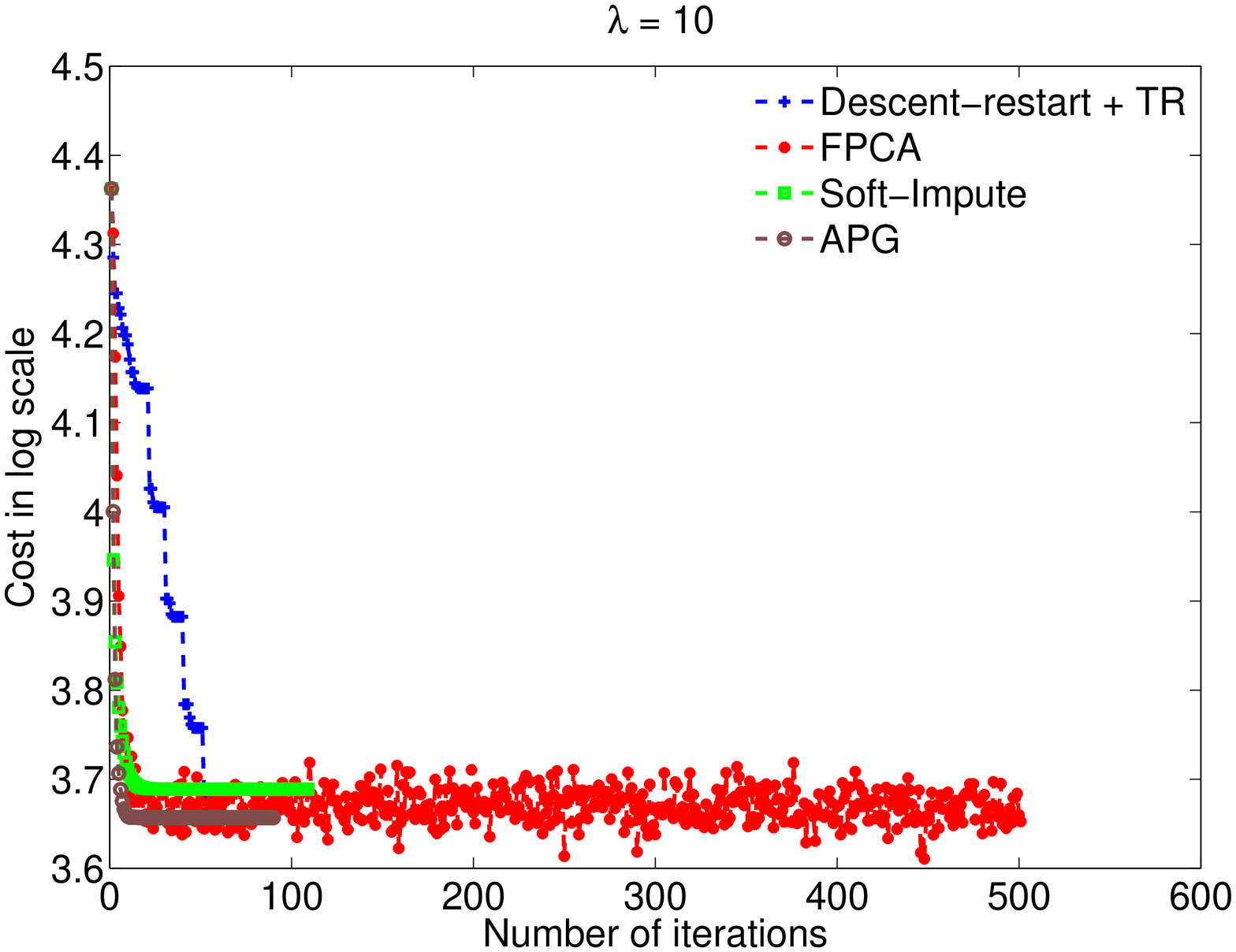}
\end{minipage}
\caption{Convergence behavior of different algorithms for different values of $\lambda$. The algorithms compared here do not use any acceleration heuristics.}
\label{fig:matrix_completion_cost_example}
\end{figure*}
The plots are shown in Figure \ref{fig:matrix_completion_cost_example}. The convergence behavior of FPCA is greatly affected by $\lambda$. It has a slow convergence for a small $\lambda$ while for a larger $\lambda$, the algorithm fluctuates. SOFT-IMPUTE has a better convergence in all the three cases, however, the convergence suffers when a more accurate solution is sought. \change{The performance of APG is robust to the change in values of $\lambda$. For moderate accuracy it outperforms all other algorithms. However, when a higher accuracy is sought it takes a large number of iterations. Descent-restart + TR, on the other hand, outperforms others in all the cases here with minimal number of iterations.}

\subsubsection*{Convergence test}
To understand the convergence behavior of different algorithms involving different optimization problems, we look at the evolution of the training error \cite{cai10a, mazumder10a} defined as
\begin{equation}\label{eq:training_error}
\mathrm{Training\ error} =  \| \mat{W}\odot (  \widetilde{\mat{X}}  -  \mat{X}   ) \|_F ^2,
\end{equation}
with iterations. We generate a $150 \times 300$ random matrix of rank $10$ under standard assumptions. \change{The over-sampling ratio is kept at $5$ with slightly less $50\%$ of the entries being observed.} The algorithms Descent-restart + TR, FPCA and SOFT-IMPUTE (Soft-I) are initialized from the same iterate with a fixed $\lambda$. We fix $\lambda = 1\times 10^{-5}$ as it gives a good reconstruction to compare algorithms. For SVT we use the initialization including $\tau= 5\sqrt{nm}$ and a step size of $ \frac{1.2}{f}$ as suggested in the paper \cite{cai10a} where $f$ is the fraction of known entries. The algorithms are stopped when the variation or relative variation of Training error (\ref{eq:training_error}) is less than $1\times 10^{-10}$. The maximum number of iterations is set at $500$. The rank incrementing procedure of our algorithm is stopped when the relative duality gap falls below $1\times 10^{-5}$.

\change{APG has a fast convergence but the performance slows down later. Consequently, it exceeds the maximum limit of iterations. Similarly, SOFT-IMPUTE converges to a different solution but has a faster convergence in the initial phase (for iterations less than $60$). FPCA and Descent-restart + TR converge faster at a later stage of their iterations. Descent-restart + TR initially sweeps through ranks until arriving at the optimal rank where the convergence is accelerated owing to the trust-region algorithm.}  
%
%
%
\begin{figure*}[ht]
\center
\includegraphics[scale = 0.30]{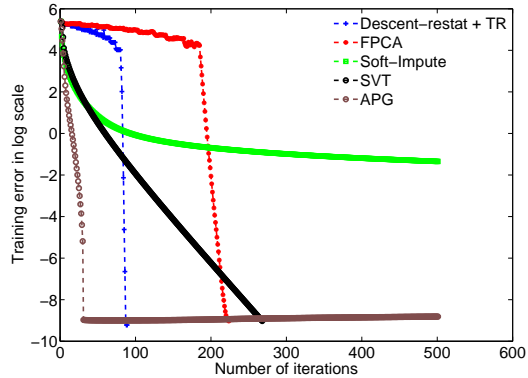}
\caption{Convergence behavior of different algorithms in terms of minimizing Training error (\ref{eq:training_error}).}
\label{fig:matrix_completion_train_error_example}
\end{figure*}

\subsubsection*{Scaling test}
To analyze the scalability of these algorithms to larger problems we perform a test where we vary the problem size $n$ from $200$ to $2200$. \change{The reason for choosing a moderate value of $n$ is that large-scale implementations of SVT, FPCA and Soft-Impute are unavailable from authors' webpages.} For each $n$, we generate a random matrix of size $n \times n$ of rank $r = 10$ under standard assumptions. Each entry is observed with uniform probability of $f = \frac{4r\rm{log}_{10}(n)}{n}$ \cite{keshavan09b}. The initializations are chosen as in the earlier example i.e., $\lambda = 10^{-5}$. We note the time and number of iterations taken by the algorithms until the stopping  criterion is satisfied or when the number of iterations exceed $500$. The stopping criterion is same as the one used before for comparison, when the absolute variation or relative variation of Training error (\ref{eq:training_error}) is less than $ 10^{-10}$. Results averaged over $5$ runs are shown in Figure \ref{fig:matrix_completion_scaling_example}. \change{We have not shown the plots for SOFT-IMPUTE and APG as in all the cases  either they did not converge in $500$ iterations or took much more time than the nearest competitor. }
\begin{figure*}[ht]
\begin{minipage}{0.5\textwidth}
\includegraphics[scale = 0.30]{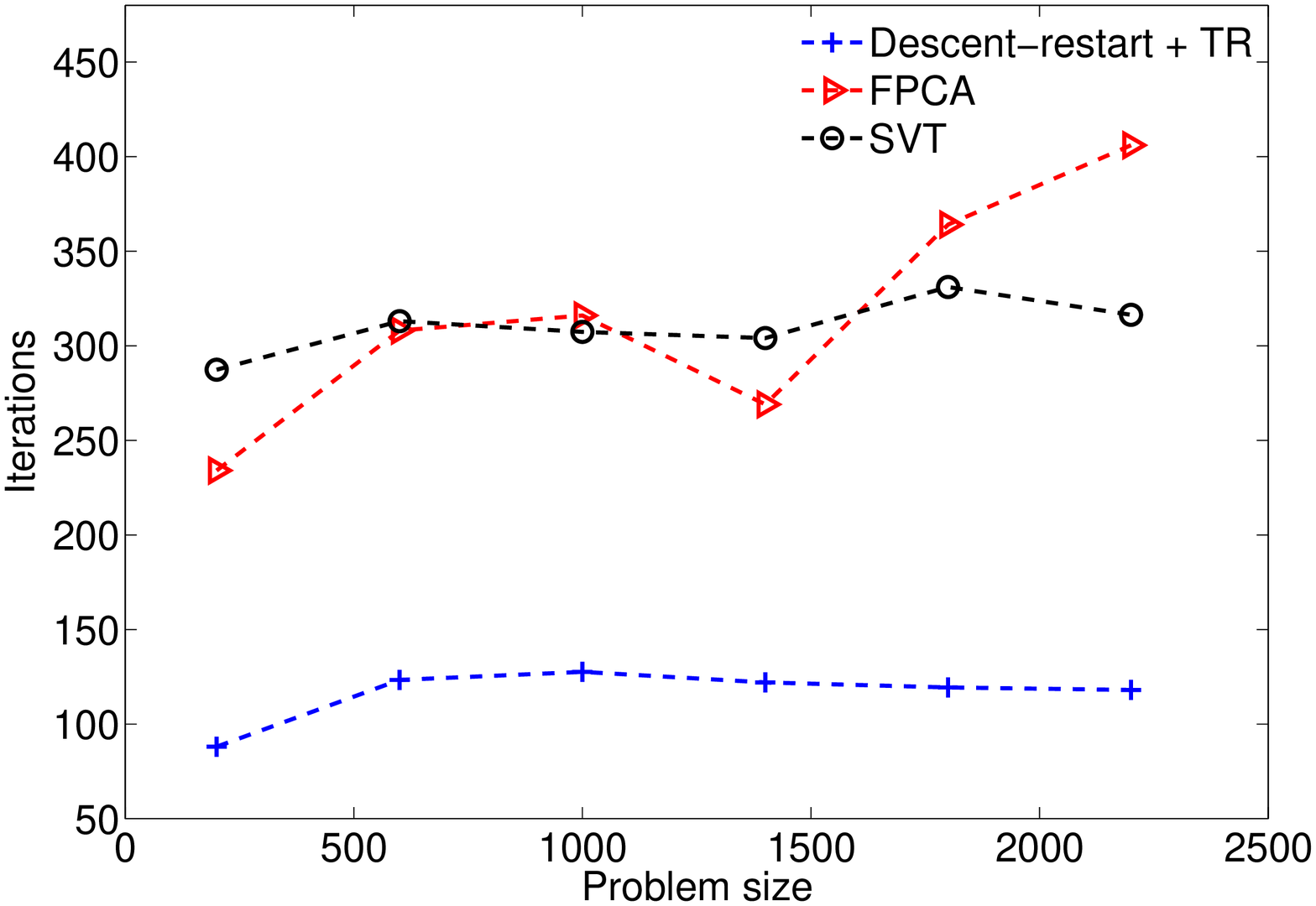}
\end{minipage}
\begin{minipage}{0.5\textwidth}
\includegraphics[scale = 0.30]{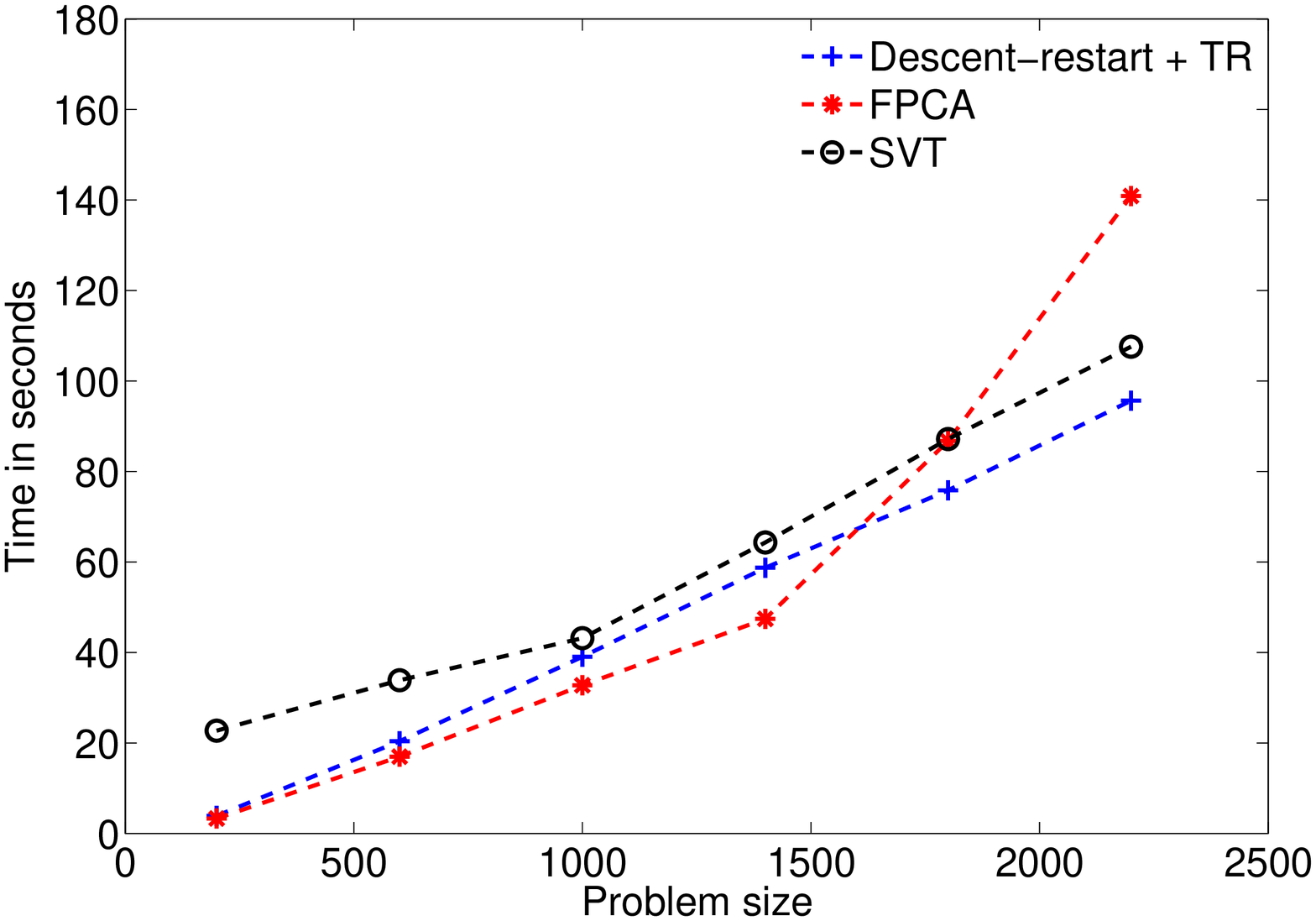}
\end{minipage}

\begin{minipage}{0.25\textwidth}  
\centering
   \scriptsize
   \begin{tabular}{|c|c|c|c|c|c|c|} \hline   
      $n$ &$200$ & $600$ & $1000$ &$1400$ & $1800$ & $2200$      \\ \hline
      $|\mathcal{W}|$ & $18409 $ & $ 66676  $ &$ 120000  $ & $176184$ & $234380$ & $294134$   \\ \hline 
      $f$ &$0.46$ & $0.19$ & $0.12$ &$0.09 $ & $0.07$ & $0.06$  \\ \hline
      $\rm{OS}$ &$4.7$ & $5.6$ & $6.0$ &$6.3$ & $6.5$ & $6.7$  \\ \hline
   \end{tabular}
\end{minipage}
\caption{Analysis of the algorithms on randomly generated datasets of rank $10$ with varying fractions of missing entries. SVT, FPCA and Descent-restart + TR have similar performances but Descent-restart + TR usually outperforms others.}
\label{fig:matrix_completion_scaling_example}
\end{figure*}

Below we have shown two more case studies where we intend to show the numerical scalability of our framework to a large scale instance. The first one involves comparisons with fixed-rank optimization algorithms. The second case is a large scale comparison with APGL (the accelerated version of APG). We consider the problem of completing a $50000 \times 50000$ matrix $\widetilde{\mat{X}}$ of rank $5$. The over-sampling ratio OS is $8$ implying that  $0.16\%$ ($3.99\times {10}^6$) of entries are randomly and uniformly revealed. The maximum number of iterations is fixed at $500$.

\subsubsection*{Fixed-rank comparison}
\change{Because our algorithm uses a fixed-rank approach for the fixed-rank sub-problem, it is also meaningful to compare its performance with other fixed-rank optimization algorithms. However, a rigorous comparison with other algorithms is beyond the scope of the present paper. We refer to a recent paper \cite{mishra12a} that deals with this question in a broader framework. Here we compare with two set-of-the-art algorithms that are LMaFit \cite{wen12a} and LRGeom (trust-region implementation) \cite{vandereycken13a}. LMaFit is an alternating minimization scheme with a different factorization for a fixed-rank matrix. It is a tuned-version of the Gauss-Seidel non-linear scheme and has a superior time complexity per iteration. LRGeom is based on the embedded geometry of fixed-rank matrices. This viewpoint allows to simplify notions of moving on the search space. We use their trust-region implementation. The geometry leads to efficient guess of the optimal stepsize in a search direction.} \change{Plots in Figure \ref{fig:ubv_vs_lmafit} show a competitive performance of our trust-region scheme with respect to LMaFit. Asymptotically, both our trust-region scheme and LRGeom perform similarly but LRGeom performs much better in the initial few iterations}.
\begin{figure}[ht]
\begin{minipage}{0.5\textwidth}
\includegraphics[scale = 0.30]{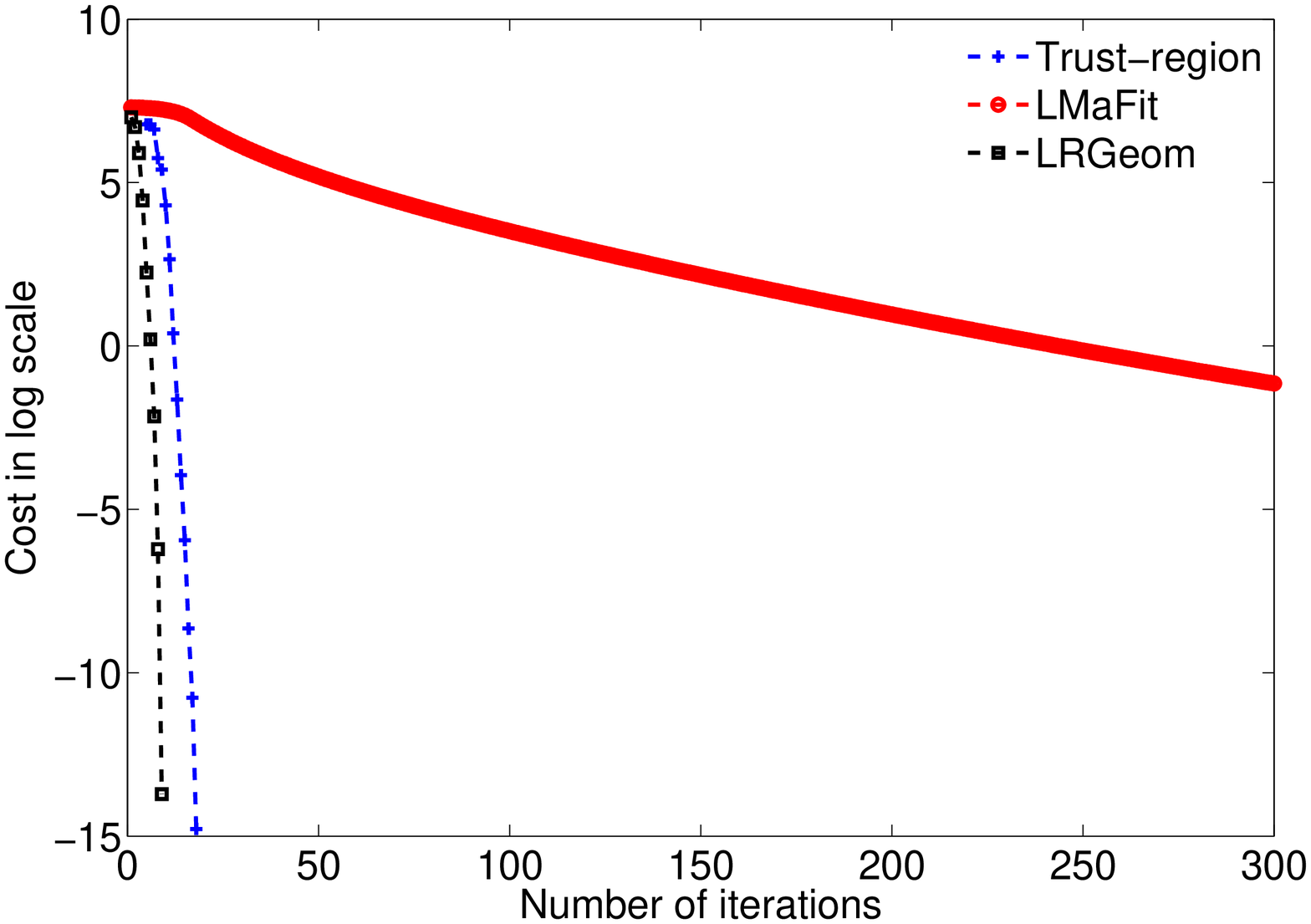}
\end{minipage}
\begin{minipage}{0.5\textwidth}
\includegraphics[scale = 0.30]{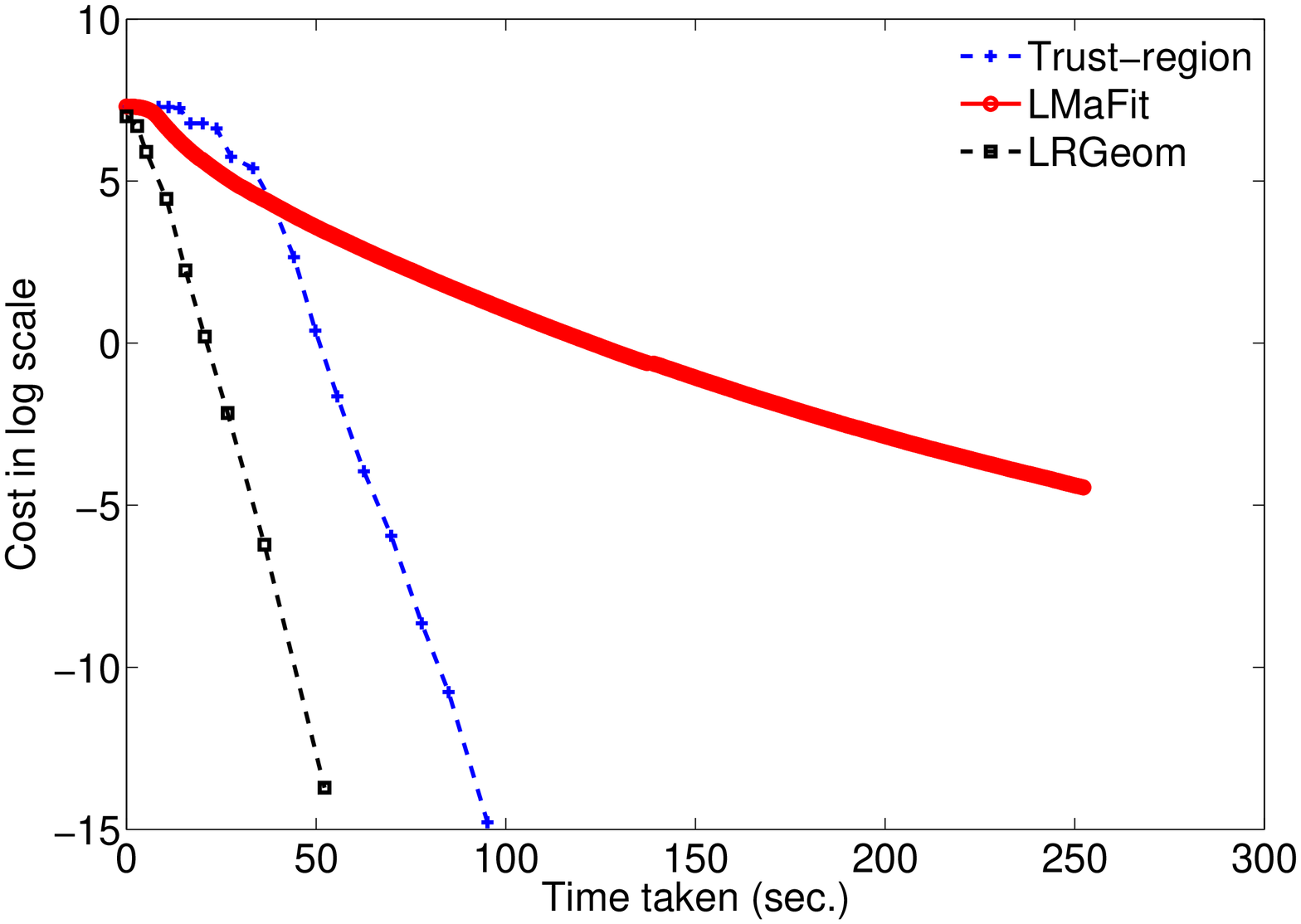}
\end{minipage}
\caption{Rank $5$ completion of $50000 \times 50000$ matrix with ${\rm OS} = 8$. All the algorithms are initialized by taking $5$ dominant SVD of sparse $\widetilde{\mat{X}}$ as proposed in \cite{keshavan10a}. Algorithms are stopped when the objective function  below a threshold, $\| \mat{W} \odot ( \widetilde{\mat{X}} - \mat{X}) \|_F ^2 \leq 10^{-10}$. The proposed trust-region scheme is competitive with LMaFit for large scale problems. Although LMaFit has a superior time complexity per iteration but its convergence seems to suffer for large-scale problems. With respect to LRGeom, the performance is poorer although both have a similar asymptotic rate of convergence.}
\label{fig:ubv_vs_lmafit}
\end{figure}
 
\subsubsection*{Comparison with APGL}
\change{APG has a better iteration complexity than other class optimization algorithms. However, scalability of APG by itself to large dimensional problems is an issue. The principal bottleneck is that the ranks of the intermediate iterates seem to be uncontrolled and only asymptotically, a low-rank solution is expected. To circumvent this issue, an accelerated version of APG called APGL is also proposed in \cite{toh10a}. APGL is APG with three additional heuristics: \emph{continuation} (a sequence of parameters leading to the target $\lambda$), \emph{truncation} (hard-thresholding of ranks by projecting onto fixed-rank matrices) and line-search technique for estimating the \emph{Lipschitz} constant. We compare our algorithm with APGL. The algorithms are stopped when either absolute variation or relative variation of the objective function falls below $10^{-10}$. For our algorithm, the trust-region algorithm is also terminated with the same criterion. In addition, the rank-one update is stopped when the relative duality gap is below $10^{-5}$.}

\change{To solve (\ref{eq:general_formulation}) for a fixed $\lambda  = \bar{\lambda}$ APGL proceeds through a sequence of values for $\lambda$ such that $\lambda_k = {\rm max} \{ 0.7\lambda_{k-1}, \bar{\lambda}\}$ where $k$ is the iteration count of the algorithm. Initial $\lambda_0$ is set to $2\| \mat{W} \odot \widetilde{\mat{X}}  \|_{op}$.} \change{We also follow a similar approach and create a sequence of values. A decreasing sequence is generated leading to $\bar{\lambda}$ is by using the recursive rule, $\lambda_{i} =\frac{\lambda_{j-1}}{ 2}$ when $ \lambda_{j-1} > 1$ and $\lambda_{i} =\frac{\lambda_{j-1}}{100}$ otherwise until $\lambda_{j-1} < \bar{\lambda}$.}
\change{Initial $\lambda_0$ is set to $\|\mat{W} \odot \widetilde{\mat{X}}  \|_{op}$. For $\lambda_j \not = \bar{\lambda}$ we also relax the stopping criterion for the trust-region to $10^{-5}$ as well as stopping the rank-one increment when relative duality gap is below $1$ as we are only interested for an accurate solution for $\lambda = \bar{\lambda}$.}
\begin{figure}[ht]
\subfigure {
\includegraphics[scale = 0.30]{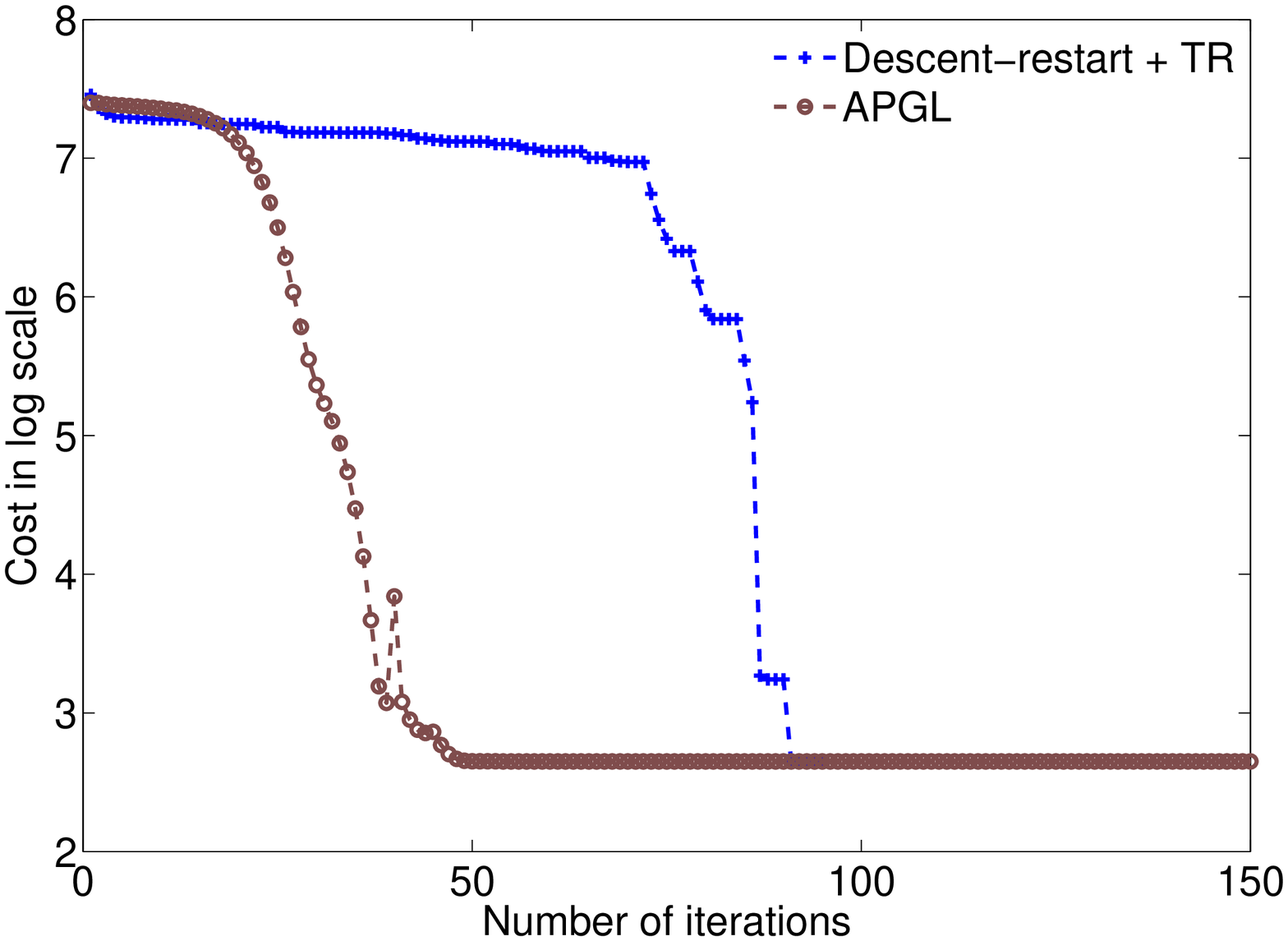}
}
\subfigure{
\includegraphics[scale = 0.30]{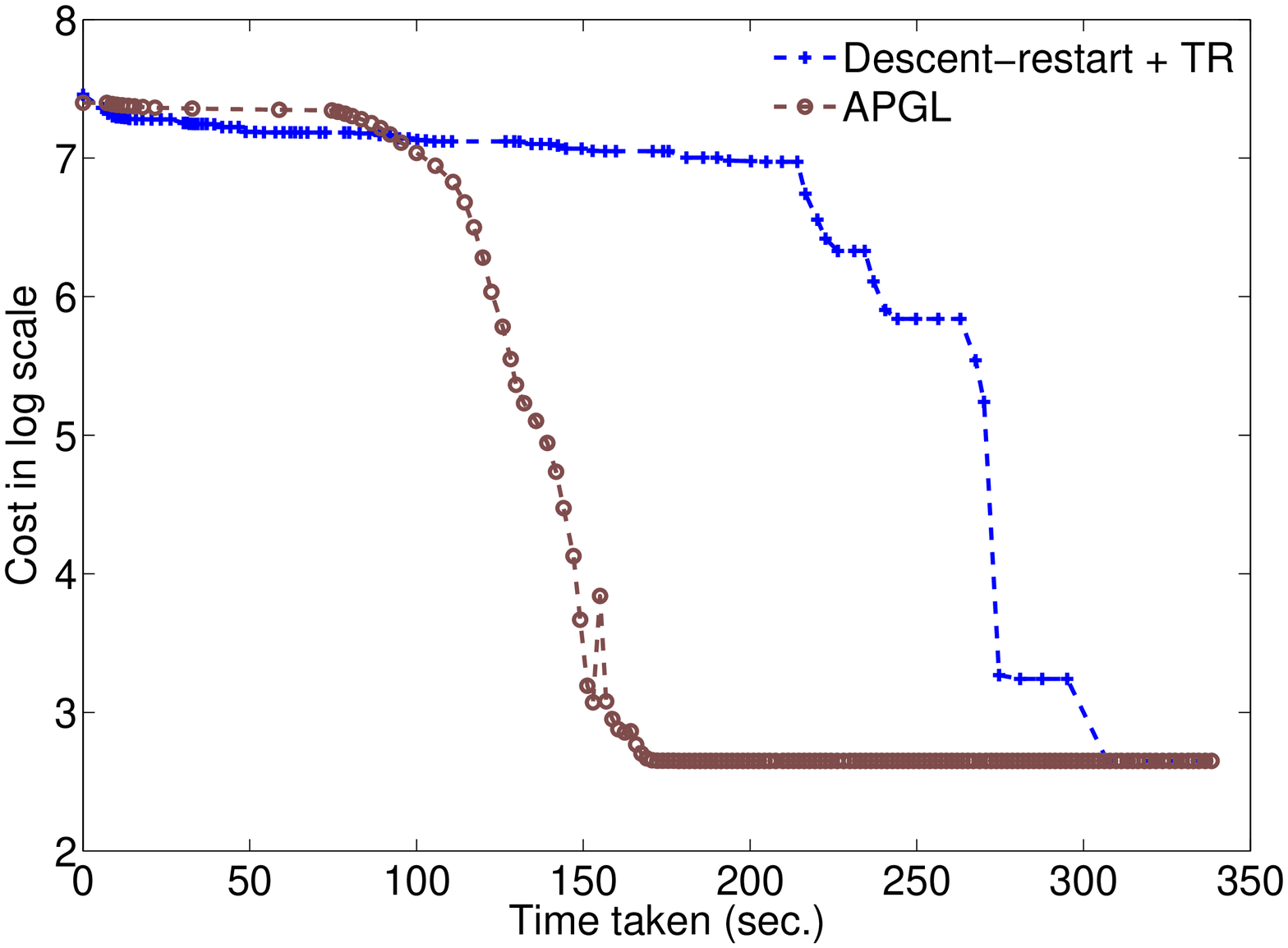}
}
\caption{A large-scale instance of rank $5$ completion of $50000 \times 50000$ matrix with ${\rm OS} = 8$. ${\lambda} = 2\|\mat{W} \odot \widetilde{\mat{X}}  \|_{op} / 10^{5}$ as suggested in \cite{toh10a}. The proposed framework is competitive for very low-ranks and when a high accuracy is sought. However, we spend considerable time in just traversing through ranks before arriving at the optimal rank.}
\label{fig:ubv_vs_apgl}
\end{figure}


\change{We compete favorably with APGL in large scale problems for very low-ranks and when a higher accuracy is required. However, as the rank increases, APGL performs better. This is not surprising as our algorithm traverses all ranks, one by one before arriving at the optimal rank. In the process we spend a considerable effort in \emph{just} traversing ranks. This approach is most effective only when computing in the full regularization path.  Also for moderate accuracy, APGL performs extremely well. However, the better performance of APGL significantly relies on heuristics like continuation and truncation. The truncation heuristic allows the APGL algorithm to approximate an iterate by low and fixed-rank iterate. On the other hand, we strictly move in the low-rank space. This provides an efficient way to compute the full regularization path using a predictor-corrector strategy.}

\subsection*{Comments on matrix completion algorithms}
We summarize our observations in the following points.
\begin{itemize}
\item{ The convergence rate of SOFT-IMPUTE is greatly dependent on the computation of singular values. For large scale problems this is a bottleneck and the performance is greatly affected. However, in our experiments, it performs quite well within a reasonable accuracy as seen in Figure \ref{fig:matrix_completion_cost_example} and Figure \ref{fig:matrix_completion_train_error_example}.}

\item{SVT, in general, performs quite well on random examples. The choice of the fixed step size and regularization parameter $\tau$, however, affect the convergence speed of the algorithm \cite{ma11a, mazumder10a}. }

\item{FPCA has a superior numerical complexity per iteration owing to an approximate singular value decomposition \cite{ma11a}. But the performance suffers as the regularization parameter $\lambda$ is increased as shown in Figure \ref{fig:matrix_completion_cost_example}.}

\item \change{APG has a better iteration complexity than the others and is well-suited when a moderate accuracy is required (Figure \ref{fig:matrix_completion_cost_example} and Figure \ref{fig:matrix_completion_train_error_example}). As the ranks of the intermediate iterates are not necessarily low, scalability to large dimension is an issue. The accelerated version APGL does not suffer from this problem and performs very well for large dimensions.}

\item \change{In all the simulation studies on random examples Descent-restart+TR has shown a favorable performance on different benchmarks. In particular our framework is well suited when the optimal solution is low-rank and when one needs to compute the regularization path. Moving strictly on the low-rank space makes it possible to go beyond the standard warm-restart approach to compute the regularization path.}
\end{itemize}

\subsection{Multivariate linear regression}
Given matrices $\mat{Y}\in \mathbb{R}^{n \times k}$ (response space) and $\mat{X} \in \mathbb{R}^{n\times q}$ (input data space), we seek to learn a weight/coefficient matrix $\mat{W} \in\mathbb{R}^{q \times k}$ that minimizes the \emph{loss} between $\mat{Y}$ and $\mat{XW}$ \cite{yuan07a}. Here $n$ is the number of observations, $q$ is the number of predictors and $k$ is the number of responses. One popular approach to multivariate linear regression problem is by minimizing a \emph{quadratic loss} function. Note that in various applications \emph{responses} are related and may therefore, be represented with much fewer coefficients. From an optimization point to view this corresponds to finding a low-rank coefficient matrix. The papers \cite{yuan07a, amit07a}, thus, motivate the use of trace norm regularization in the following optimization problem formulation, defined as, 
\[
\min\limits_{\mat{W} \in \mathbb{R}^{q \times k}} \quad \| \mat{Y} - \mat{XW}\|_F^2 + \lambda \|\mat{W} \|_*.
\]
(Optimization variable is $\mat{W}$.) Although the focus here is on the quadratic loss function, the framework can be applied to other smooth loss functions as well. Other than the difference in the dual variable $\mat{S}$ and $\mat{S}_*$, the rest of the computation of gradient and its directional derivative in the Euclidean space is similar to that of the low-rank matrix completion case.
$
\mat{S} = 2 (   \mat{X} ^T \mat{X} \mat{W} - \mat{X}^T \mat{Y}   ) \ \ \rm{and}\ \ 
\mat{S}_* = \D_{(\mat{U}, \mat{B}, \mat{V})} \mat{S} [\mat{Z}] =  2( \mat{X}^T \mat{X} (\mat{Z}_{\mat{U}}\mat{B}\mat{V}^T +  \mat{U} \mat{Z}_{\mat{B}} \mat{V}^T  + \mat{UB} \mat{Z}_{\mat{V}}^T   ))
$
where \change{the rank of $\mat{W}$ is $p$ and} $\mat{W} = \mat{UBV}^T$.

The numerical complexity per iteration is dominated by the numerical cost to compute $\bar{\phi}(\mat{U}, \mat{B}, \mat{V})$, $\mat{S}$ and terms like $\mat{SVB}$. The cost of computing $\bar{\phi}$ is of $O(nqp + nkp + kp^2 + nk)$ and $\mat{SVB}$ is $O(q^2p + qkp + kp^2 )$.  And that of full matrix $\mat{S} $ is $ O(q^2p + qkp + kp^2 )$. From a \emph{cubic} numerical complexity of $O(q^2k)$ per iteration (using the full matrix $\mat{W}$ ) the low-rank factorization reduces the numerical complexity to $O(q^2p + qkp)$ which is \emph{quadratic}. Note that the numerical complexity per iteration is linear in $n$.

\subsection*{Fenchel dual and duality gap for multivariate linear regression}
For the multivariate linear regression problem $\mathcal{A}(\mat{W}) = \mat{XW}$ and therefore, we can define $\psi$ such that $f(\mat{W}) = \psi (\mat{XW})$. So, $\mathcal{A}^*(\mat{\eta}) = \mat{X}^T \eta$. The dual candidate $\mat{M}$ is defined by $\mat{M} =  \min(1, \frac{\lambda }{\sigma_\psi}) \Grad \psi$ where $\Grad \psi ( \mat{X W}) = 2 (\mat{XW} - \mat{Y})$ and $\sigma_\psi$ is the dominant singular value of $\mathcal{A}^*(\Grad \psi) = \mat{X}^T \Grad \psi$. The Fenchel dual $\psi^*$ (after few more steps) can be computed as $\psi^*(\mat{M}) = \trace(\mat{M}^ T \mat{M})/4 + \trace(\mat{M}^T \mat{Y})$. \change{Finally, the duality gap is $f(\mat{W}) + \lambda \| \mat{W}\|_* + \psi^*(\mat{M})$}. As we use a low-rank factorization of $\mat{W}$, i.e., $\mat{W} = \mat{UBV}^T$ the numerical complexity of finding the duality gap is dominated by numerical cost of computing $\psi^*(\mat{M})$ which is also of the order of the cost of computing $\bar{\phi}(\mat{U}, \mat{B}, \mat{V})$. Numerical complexity of computing $\mat{M} $ is $ O(nqp + nkp + kp^2)$ and of $\psi^*(\mat{M})$ is $O(nk)$.

\subsubsection{Regularization path for multivariate linear regression}
An input data matrix $\mat{X}$ of size $5000 \times 120 $ is randomly generated according to a Gaussian distribution with zero mean and unit standard deviation. The response matrix $\mat{Y}$ is computed as $\mat{X}\mat{W}_*$ where $\mat{W}_*$ is a randomly generated coefficient matrix of rank $5$ matrix and size $120 \times 100$. We randomly split the observations as well as responses into \emph{training} and \emph{testing} datasets in the ratio $70/30$ resulting in $\mat{Y}_{\rm{train}}/\mat{Y}_{\rm{test}}$ and $\mat{X}_{\rm{train}}/\mat{X}_{\rm{test}}$. A Gaussian white noise of zero mean and variance $\sigma_{\rm{noise}}^2$ is added to the training response matrix $\mat{Y_{\rm{train}}}$ resulting in $\mat{Y}_{\rm{noise}}$. We learn the coefficient matrix $\mat{W}$ by minimizing the \emph{scaled} cost function, i.e.,
\[
\min\limits_{\mat{W} \in \mathbb{R}^{q \times k}} \quad \frac{1}{nk}\| \mat{Y}_{\rm{noise}} - \mat{X_{\rm{train}}W}\|_F^2 + \lambda \|\mat{W} \|_*,
\]
where $\lambda$ is a regularization parameter. We validate the learning by computing the root mean square error (rmse) defined as 
\[
{\rm Test\ rmse} = \sqrt{ \frac{1}{ n_{\rm{test}} k}   \| \mat{Y}_{\rm{test}} - \mat{X_{\rm{test}}W}\|_F^2  }
\]
where $n_{\rm{test}}$ is the number of test observations. Similarly, the signal to noise ratio (SNR) is defined as $\sqrt{\frac{\| \mat{Y}_{\rm{train}} \|_F^2}{\sigma_{\rm{noise}}^2}}$.

We compute the entire regularization path for four different SNR values. The maximum value of $\lambda$ is fixed at $10$ and the minimum value is set at $1\times 10^{-5}$ with the reduction factor $\gamma = 0.95$. Apart from this we also put the restriction that we only fit ranks less than $30$. The solution to an optimization problem for $\lambda_{j}$ is claimed to have been obtained when either the duality gap falls below $1\times 10^{-2}$ or the relative duality gap falls below $1\times 10^{-2}$ or $\sigma_1 - \lambda$ is less than $1\times 10^{-2}$. Similarly, the trust-region algorithm stops when relative or absolute variation of the cost function falls below $1\times10^{-10}$. The results are summarized in Figure \ref{fig:multivariate_regression}.

\begin{figure}[ht]
\begin{minipage}{0.5\textwidth}
\includegraphics[scale = 0.30]{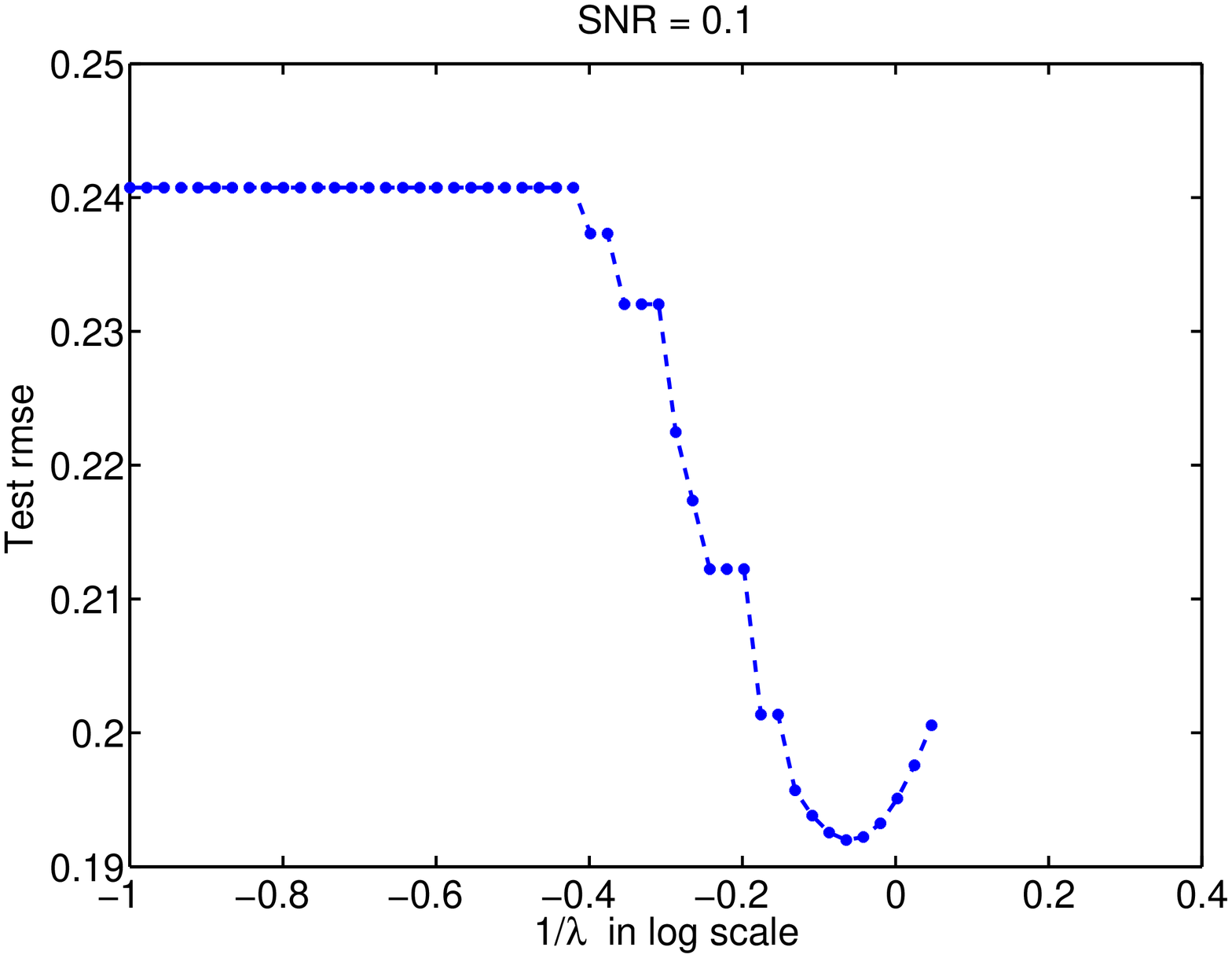}
\end{minipage}
\begin{minipage}{0.5\textwidth}
\includegraphics[scale = 0.30]{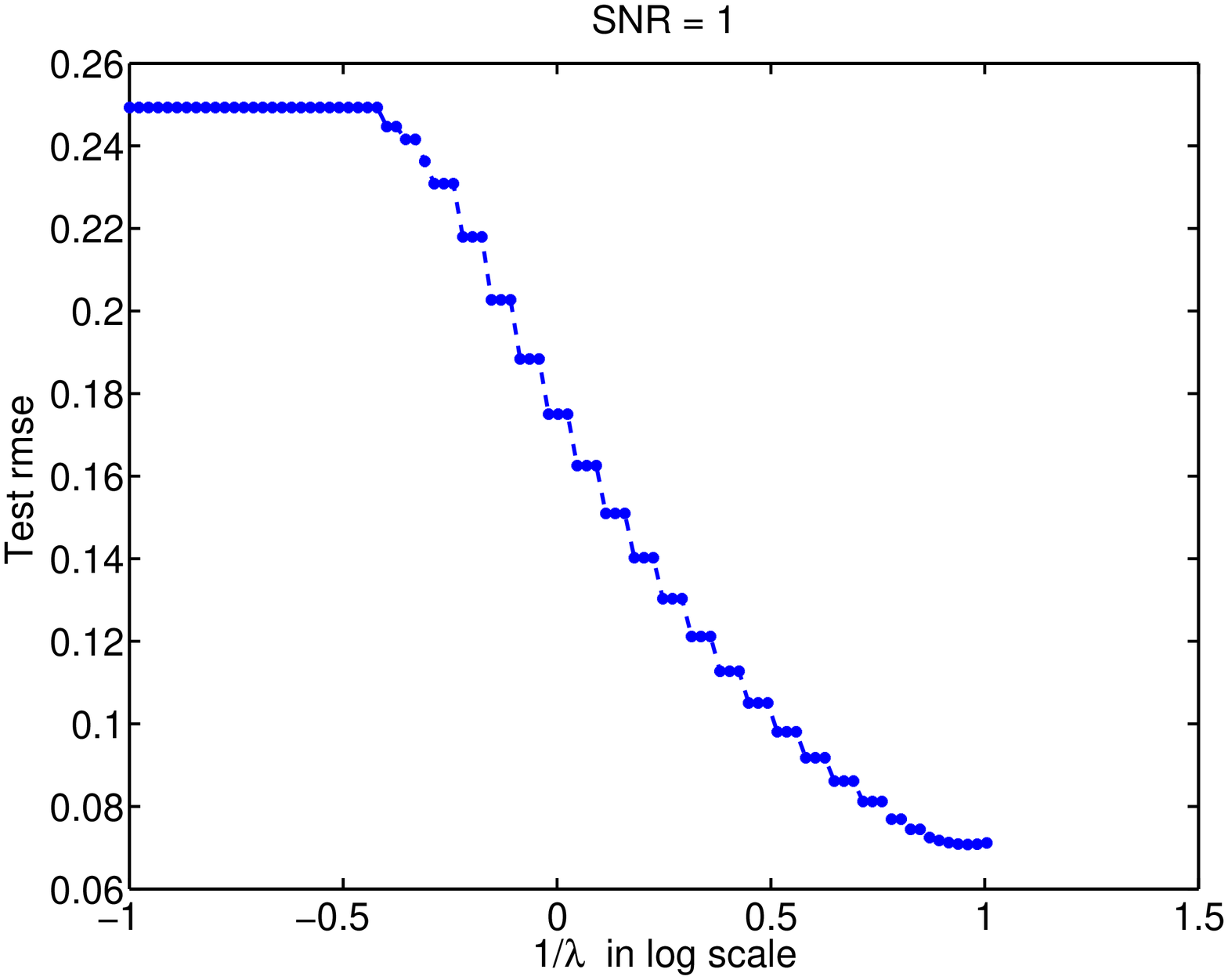}
\end{minipage}
\begin{minipage}{0.5\textwidth}
    \includegraphics[scale = 0.30]{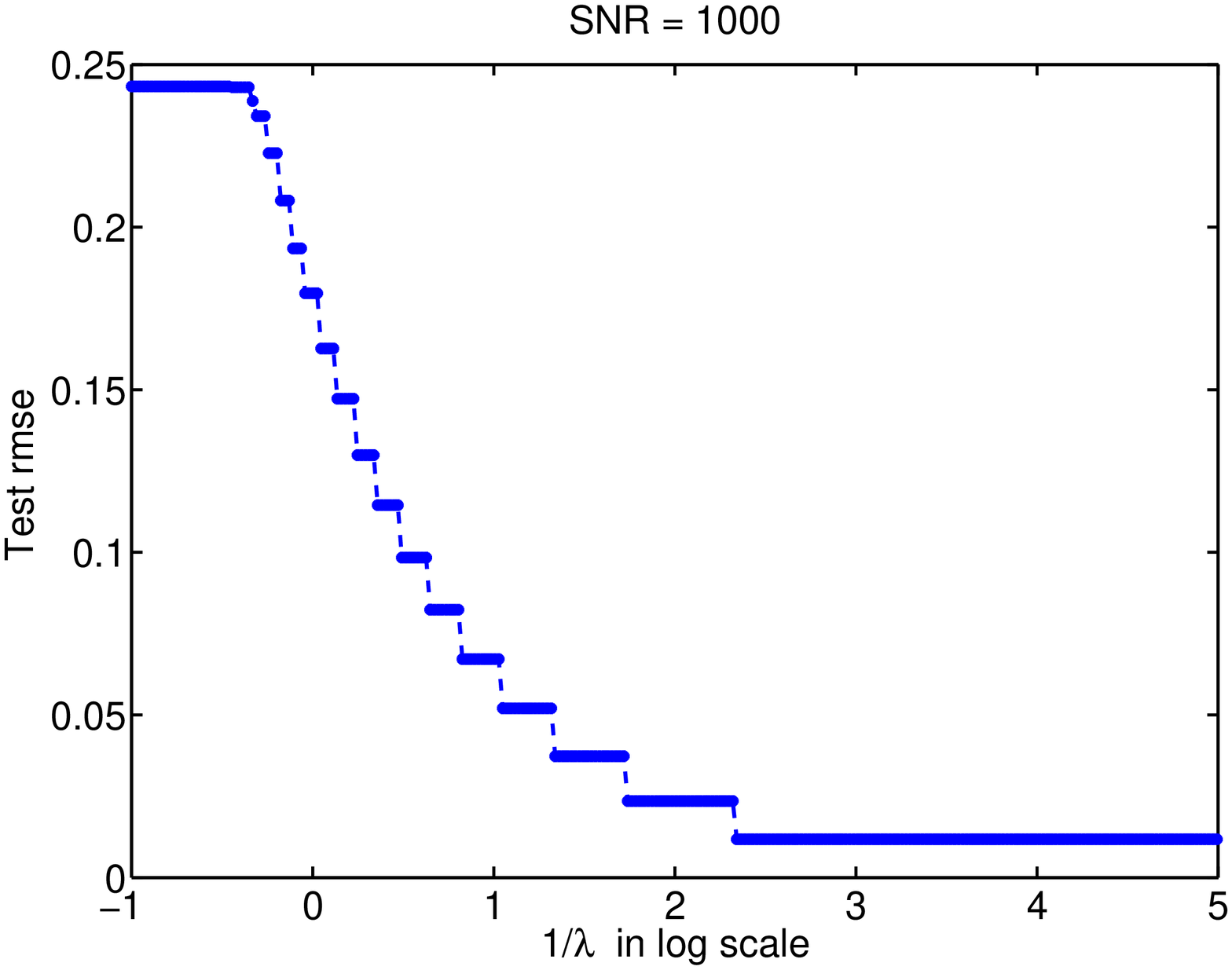}
\end{minipage}
\begin{minipage}{0.25\textwidth}
  \scriptsize
   \begin{tabular}{|c|c|c|c|c|c|} \hline   
      SNR     & Test rmse & Rank & $\lambda$      \\ \hline
      $0.1$ & $0.192$ & $15$ & $ 1.16$  \\ \hline 
       $1$ & $0.071$ & $22$ & $ 0.109$  \\ \hline 
      $1000$ & $0.012$ & $5$ & $ 0.005$  \\ \hline
       $1\times 10^8$ & $7.87\times 10^{-3}$ & $5$ & $ 0.002$  \\ \hline 
 
         \end{tabular}
\end{minipage}
\caption{ Regularization path for multivariate linear regression with various SNR values. Results are averaged over 5 random $70/30$ splits. }
\label{fig:multivariate_regression}
\end{figure}


\section{Conclusion}
Three main ideas have been presented in this paper. First, we have given a framework to solve a general trace norm minimization problem (\ref{eq:general_formulation}) with a sequence of increasing but fixed-rank non-convex problems (\ref{eq:factorized_formulation}). We have analyzed the convergence criterion and duality gap which are used to monitor convergence to a solution of the original problem. The duality gap expression was shown numerically tractable even for large problems thanks to the specific choice of the low-rank parameterization. We have also given a way of incrementing the rank while simultaneously ensuring a decrease of the cost function. This may be termed as a \emph{descent-restart} approach. The second contribution of the paper is to present a second-order trust-region algorithm for a general rank-$p$ (fixed-rank) optimization in the quotient search space $\Stiefel{p}{n}\times \PD{p}\times \Stiefel{p}{m}/\OG{p}$ equipped with the natural metric $\bar{g}$ (\ref{eq:metric}). The search space with the metric $\bar{g}$ has the structure of a Riemannian submersion \cite{absil08a}. We have used manifold-optimization techniques \cite{absil08a} to derive the required geometric objects in order to devise a second-order algorithm. With a proper parameter tuning the proposed trust-region algorithm guarantees a quadratic rate of convergence. The third contribution of the paper is to develop a predictor-corrector algorithm on the quotient manifold where the predictor step is along the first-order approximation of the geodesic. The corrector step is achieved by initializing the descent-restart approach from the predicted point. The resulting performance is superior to the warm-restart approach.

These ideas have been applied to the problems of low-rank matrix completion and multivariate linear regression leading to encouraging numerical results.

\appendix 

\section{Proofs}

\subsection{Derivation of first-order optimality conditions of (\ref{eq:foc_factorized_formulation})} \label{appendix:foc}
We derive the gradient $\grad_{\bar x}\bar{ \phi}$ in the total space $\overline{\mathcal{M}}_p$ with the metric (\ref{eq:metric}) using (\ref{eq:riemannian_gradient}) at point ${\bar x} = (\mat{U}, \mat{B}, \mat{V})$. First, we compute $\Grad_{\bar x}$ of $\bar{\phi}$ in the Euclidean space $\mathbb{R}^{n \times r} \times \mathbb{R}^{r \times r} \times \mathbb{R}^{m \times r}$. The matrix representation of $\Grad_{\bar x}$ is $
(\Grad_{\mat{U}} \bar{\phi}, \Grad_{\mat{B}} \bar{\phi}, \Grad_{\mat{V}} ) = (\mat{SVB}, \mat{U}^T \mat{SV} + \lambda \mat{I}, \mat{S}^T \mat{UB} )
$ which leads to the expression
\begin{equation*}
\begin{array}{llll}
\grad_{\mat U} {\bar \phi} &= & (\mat{S}\mat{VB} - \mat{U}\Sym(\mat{U}^T \mat{SVB}), \
\grad_{\mat B}{\bar \phi}  =   \mat{B}\left( \Sym(\mat{U}^T \mat{SV}) + \lambda  \mat{I} \right) \mat{B} \\
\grad_{\mat V}{\bar \phi} & = & \mat{S}^T \mat{UB} -\mat{V} \Sym (\mat{V}^T \mat{S}^T \mat{UB} )
\end{array}
\end{equation*}
The conditions (\ref{eq:foc_factorized_formulation}) are obtained by $\|\grad_{\bar x} \bar{\phi}\|_{g_{\bar x}} =0$.

\subsection{Proof of Proposition (\ref{prop:convergence})}\label{appendix:foc_convex}
From the characterization of sub-differential of trace norm \cite{recht10a} we have the following
\begin{equation}\label{eq:subdifferential_trace_norm}
\begin{array}{lll}
\partial \|\mat{X} \|_* =  \{   \mat{UV}^T + \mat{W}\  | \ &  \mat{W}  {\rm \ and\ } \mat{X} {\rm \ have \ orthogonal \ column \ and\ row\ spaces,}\\  
 & \mat{W} \in \mathbb{R}^{n \times m} \ {\rm and} \ \|\mat{W} \|_{op} \leq 1  \}  
\end{array}
\end{equation}
where $\mat{X} = \mat{UBV}^T$. Since $\mat{X} = \mat{UBV}^T$ is also a stationary point for the problem (\ref{eq:factorized_formulation}), the conditions (\ref{eq:foc_factorized_formulation}) are satisfied including $\Sym(\mat{U}^T \mat{SV}) + \lambda  \mat{I} = \mat{0}$. From the properties   of a matrix norm we have
\[
\begin{array}{llll}
&\lambda  \mat{I} &=& -\Sym(\mat{U}^T \mat{SV}) \\
\Rightarrow & \lambda & = & \| \Sym(\mat{U}^T \mat{SV})   \|_{op} \leq  \| \mat{U}^T \mat{SV}   \|_{op} \leq \mat{S}   \|_{op}. \\
\end{array}
\]
Equality holds if and only if when $\mat{U}$ and $\mat{V}$ correspond to the dominant row and column subspace of $\mat{S}$, i.e., if $\mat{S} = -\lambda \mat{UV}^T + \mat{U_{\perp}\Sigma V_{\perp}}^T$ where $\mat{U}^T \mat{U_{\perp}} = \mat{0}$, $\mat{V}^T \mat{V_{\perp}} = \mat{0}$, $\mat{U_{\perp}}\in \Stiefel{n-p}{n}$, $\mat{V_{\perp}}\in \Stiefel{m-p}{m}$ and $\mat{\Sigma}$ is diagonal matrix with positive entries with $\| \mat{\Sigma}\|_{op} \leq \lambda$. Note that this also means that  $\mat{S} \in \lambda \partial \|\mat{X} \|_*$ such that $\mat{W} = \mat{U_{\perp}\Sigma V_{\perp}}^T$ which satisfies (\ref{eq:subdifferential_trace_norm}) and the global optimality condition (\ref{eq:foc_general_formulation}). This proves Proposition (\ref{prop:convergence}).

\subsection{Proof of Proposition (\ref{prop:descent_directions})}\label{appendix:descent_directions}
Since $\mat{X} = \mat{UBV}^T$ is a stationary point for the problem (\ref{eq:factorized_formulation}) and not the global optimum of (\ref{eq:general_formulation}) by virtue of Proposition \ref{prop:convergence} we have $\|\mat{S} \|_{op} > \lambda$ (strict inequality).  We assume that $f$ is smooth and hence, let the first derivative of $f$ is Lipschitz continuous with the Lipschitz constant $L_{f}$, i.e., $\| \nabla f_{\mat X}(\mat{X}) -  \nabla f_{\mat Y}(\mat{Y}) \|_F \leq L_f \| \mat{X} - \mat{Y} \|_F$\ where $\mat{X}, \mat{Y}\in \mathbb{R}^{n \times m}$ \cite{nesterov03a}. Therefore, the update (\ref{eq:update}), $\mat{X}_{\rm +} = \mat{X} - \beta u v^T$ would result in following inequalities
\begin{equation}\label{eq:decrease_f}
\begin{array}{llll}
&f(\mat{X}_{\rm +})  &\leq& f(\mat{X}) + \langle \nabla_{\mat{X}} f (\mat{X}), \mat{X}_{\rm +} - \mat{X}  \rangle  + \frac{L_f}{2} \| \mat{X}_{\rm +} - \mat{X} \|_{F}^2 \ {\rm ( from \ \cite{nesterov03a} )} \\
&  & = & f(\mat{X}) - \beta \sigma_1 + \frac{L_f}{2}\beta^2 \\
&{\rm also} & & \\
&\|\mat{X}_{\rm +} \|_*  &\leq &  \| \mat{X} \|_* +\beta  \ {\rm ( from \ triangle\ inequality \ of \ matrix \ norm )}\\
\Rightarrow & f(\mat{X}_{\rm +}) + \lambda\|\mat{X}_{\rm +} \|_* & \leq & f(\mat{X}) + \lambda\|\mat{X} \|_* - \beta ( \sigma_1 - \lambda -  \frac{L_f}{2}\beta)
\end{array}
\end{equation}
for $\beta >0$ and $\sigma_1$ being $\| \mat{S} \|_{op}$. The maximum decrease in the cost function is obtained by maximizing $\beta ( \sigma_1 - \lambda -  \frac{L_f}{2}\beta)$ with respect to $\beta$ which gives $\beta_{\rm max} = 	\frac{\sigma_1 - \lambda}{L_f} > 0$. $\beta_{\rm max} = 0$ only at optimality. This proves the proposition.

\subsection{Proof of Proposition (\ref{prop:dual_formulation})}\label{appendix:dual_formulation}
Without loss of generality we introduce a dummy variable $\mat{Z} \in \mathbb{R}^{n \times m}$ to rephrase the optimization problem (\ref{eq:general_formulation}) as
\[
\begin{array}{llll}
\min\limits_{\mat{X}, \mat{Z}} & f(\mat{X}) + \lambda \| \mat{Z} \|_*\\
 \subject & \mat{Z} = \mat{X}. \\ 
\end{array}
\] 
The Lagrangian of the problem with dual variable $\mat{M} \in \mathbb{R}^{n \times m}$ is
$\mathcal{L}(\mat{X}, \mat{Z}, \mat{M}) = f(\mat{X}) + \lambda \| \mat{Z} \|_* + \trace(\mat{M}^T (\mat{Z} - \mat{X}))$. The \emph{Lagrangian dual} function $g$ of the Lagrangian $\mathcal{L}$ is, then, computed as \cite{boyd04a, bach11a}
\[
\begin{array}{lllll}
&g(\mat{M})  &= &\min\limits_{\mat{X}, \mat{Z}} \quad f(\mat{X})  - \trace(\mat{M}^T \mat{X}) + \trace(\mat{M}^T \mat{Z}) + \lambda \| \mat{Z}\|_*\\
\Rightarrow &g(\mat{M})  &= &\min\limits_{\mat{X}}  \quad  \{ f(\mat{X})  - \trace(\mat{M}^T \mat{X})  \} + \min\limits_{\mat{Z}} \quad  \{ \trace(\mat{M}^T \mat{Z}) + \lambda \| \mat{Z}\|_* \}
\end{array}
\]
Using the concept of dual norm of trace norm, i.e., operator norm we have 
\[
\min\limits_{\mat{Z}} \quad  \trace(\mat{M}^T \mat{Z}) + \lambda \| \mat{Z}\|_*  = 0 \quad {\rm if} \quad \| \mat{M} \|_{op} \leq \lambda
\] 
Similarly, using the concept of Fenchel conjugate of a function we have 
\[
\min\limits_{\mat{X}}  \quad f(\mat{X})  - \trace(\mat{M}^T \mat{X})  = - f^*(\mat{M})
\]
where $f^*$ is the Fenchel conjugate \cite{bach11a, boyd04a} of $f$, defined as $f^*(\mat{M}) =\\ {\rm sup}_{\mat{X} \in \mathbb{R}^{n \times m}} \left [ \trace(\mat{M}^T \mat{X})  - f(\mat{X}) \right ]$. Therefore when $\| \mat{M} \|_{op} \leq \lambda$, the final expression for the dual function is $g(\mat{M}) = - f^*(\mat{M})$\cite{bach11a} and the Lagrangian dual formulation boils down to
\[
\begin{array}{lll}
{\rm max}_{\mat M}& - f^*(\mat{M})\\
 \subject &  \| \mat{M} \|_{op} \leq \lambda.
\end{array}
\]
This proves the proposition.
\bibliographystyle{amsalpha}
\bibliography{arXiv_MMBS_trace_norm}

\end{document}